\documentclass[12pt,reqno]{amsart}
\usepackage{amsmath,amsfonts,amsthm,amssymb,color,hyperref}

\usepackage[T1]{fontenc}
\usepackage{mathtools}
\usepackage{graphicx}

\usepackage{subfigure} 

\makeatletter
     \def\section{\@startsection{section}{1}%
     \z@{.7\linespacing\@plus\linespacing}{.5\linespacing}%
     {\bfseries
     \centering
     }}
     \def\@secnumfont{\bfseries}
     \makeatother

\usepackage{graphicx}

\newcommand{\R}{\mathbb R}

\newcommand{\Z}{\mathbb Z}
\newcommand{\E}{\mathbb E}
\newcommand{\gammav}{\gamma_{\rm av}}

\newcommand{\co}{R}

\newcommand{\ep}{\varepsilon}

\newcommand{\cexpf}{2H} %tail for local itself
\newcommand{\cexps}{2H(1+\gamma)} %tail for increment
\newcommand{\cexpt}{2H} %correct exponent for \theta_n
\newcommand{\cexp}{2H} %correct exponent for log term in iterated law. Relation is \nu_1 = \nu and \nu = \nu_2 - \nu_3 \gamma

\newtheorem{theorem}{Theorem}[section]
\newtheorem{lemma}[theorem]{Lemma}
\newtheorem{proposition}[theorem]{Proposition}
\newtheorem{corollary}[theorem]{Corollary}

\theoremstyle{definition}

\theoremstyle{remark}
\newtheorem{remark}[theorem]{Remark}

\numberwithin{equation}{section}
\newcommand{\CR}{\textcolor{black}}

\begin{document}
\title[Local times]{Local times and sample path properties of the Rosenblatt process}
\author[G. Kerchev]{George Kerchev}
\address{University of Luxembourg, Department of Mathematics, Luxembourg}
 \email{george.kerchev@uni.lu}
\author[I. Nourdin]{Ivan Nourdin}
\address{University of Luxembourg, Department of Mathematics, Luxembourg}
 \email{ivan.nourdin@uni.lu}
 \author[E. Saksman]{Eero Saksman}
\address{University of Helsinki, Department of Mathematics and Statistics, Finland}
\email{eero.saksman@helsinki.fi}
\author[L. Viitasaari]{Lauri Viitasaari}
\address{Aalto University School of Business, Department of Information and Service Management, Finland}
\email{lauri.viitasaari@iki.fi}

\begin{abstract}
Let $Z = (Z_t)_{t \geq 0}$ be the Rosenblatt process with Hurst index $H \in (1/2, 1)$. We prove joint continuity for the local time of $Z$, and  establish H\"older conditions for the local time. These results are then used to study the irregularity of the sample paths of $Z$.  Based on analogy with similar known results in the case of fractional Brownian motion, we believe our results are sharp. A main ingredient of our proof is a rather delicate spectral analysis of arbitrary linear combinations of integral operators, which arise from  the representation of  the Rosenblatt process as an element in the  second chaos.  
\end{abstract}

\maketitle



\allowdisplaybreaks

\section{Introduction}

The Rosenblatt process $Z = (Z_t)_{t \geq 0}$ is a self-similar stochastic process with long-range dependence and heavier tails than those of the normal distribution. It depends on a parameter $H \in (\frac12, 1)$ which is fixed in what follows. 
The process $Z$ belongs to the family of Hermite processes that naturally arise as limits of normalized sums of long-range dependent random variables \cite{Dobrushin-Major-1979}. The first Hermite process is the fractional Brownian motion, which is Gaussian and thus satisfies many desirable properties. The Rosenblatt process is the second Hermite process: it is no longer Gaussian and was introduced by Rosenblatt in \cite{Rosenblatt-1961}.  Although less popular than fractional Brownian motion, the Rosenblatt process is attracting increasing interest in the literature, mainly due to its self-similarity, stationarity of increments\CR{,} and long-range dependence properties. See for example~\cite{Aurzada-Monch-2018, Chronopoulou-Viens-Tudor-2009,Gu-Bal-2012, Nourdin-Tran-2019}  for recent studies on different aspects of this process.

The Rosenblatt process $(Z_t)_{t \geq 0}$ admits the following stochastic representation, also known as the spectral representation:
\begin{align}
\label{eq:rosen_rep_kernel} Z_t = \int_{\mathbb{R}^2} H_t(x, y) Z_G(dx) Z_G(dy).
\end{align}
In (\ref{eq:rosen_rep_kernel}), the double Wiener-It\^o integral is taken over $x \neq \pm y$ \CR{ and}  $Z_G(dx)$ is a complex-valued random white noise with control measure $G$ satisfying  $G(tA) = |t|^{1 - H} G(A)$ for all $t \in \R$ and Borel sets $A \in \R$ and   $G(dx) = |x|^{-H} dx$\CR{. The} kernel $H_t(x, y)$ is given by
\begin{align}
\label{eq:rosen_kernel} H_t(x, y) = \frac{e^{i t (x + y) } -1 }{i (x + y)},
\end{align}
\noindent and is a complex valued Hilbert-Schmidt kernel satisfying $H_t(x, y) = H_t(y, x) = \overline{H_t(-x, -y)}$ and $\int_{\mathbb{R}^2} |H_t(x, y)|^2 G(dx) G(dy)< \infty$.
\noindent In particular, the spectral theorem applies, see~\cite{Dobrushin-Major-1979}, and allows one to write
\begin{align}
\label{eq:rosenblatt_spectral_dec} Z_t \overset{d}{=} \sum_{k =1 }^{\infty} \lambda_k (X_k^2 -1 ), 
\end{align}
\noindent where $(X_k)_{k \geq 1}$ is a sequence of independent standard Gaussian random variables and $(\lambda_k)_{k \geq 1}$ are the eigenvalues\footnote{\CR{We shall actually later work more with the singular value sequence, which for a compact self-adjoint operator coincides with sequence of the absolute values of the eigenvalues in a decreasing order,  and each one repeated according to multiplicity. Singular values have the advantage that they can be defined also for non-self-adjoint operators and satisfy very handy inequalities.}} \CR{(repeated according to the possible multiplicity)} of the \CR{self-adjoint} operator $A:L_G^2 (\mathbb{R}) \to L_G^2 (\mathbb{R})$, given by
\begin{align}
\notag (Af) (x) \coloneqq \int_{\mathbb{R}} H_t(x, -y) h(y) G(dy) = \int_{\mathbb{R}} H_t(x, -y) h(y) |y|^{-H} dy. 
\end{align}
\noindent Furthermore, $\sum_{k \geq 1} \lambda_k^2 < \infty$ and thus~\eqref{eq:rosenblatt_spectral_dec} converges.

\smallskip

\smallskip

The goal of the present paper is to study the occupation density (also known as \CR{the} local time) $L(x, I)$ of $Z$\CR{,} where $x \in \mathbb{R}$ and $I \subset [0, \infty)$ is a finite interval. 
Recall that for a deterministic function $f: \mathbb{R}_+ \to \mathbb{R}$, the occupation measure of $f$ is defined by
\begin{align}
\notag \nu(A, B) = \mu(B \cap f^{-1}(A)), 
\end{align}
\noindent where $A\subset\R$ and $B\subset\R_+$ are Borel sets and $\mu$ is the Lebesgue measure on $\mathbb{R}_+$. Observe that $\nu(A,B)$ represents the amount of time during a period $B$ where $f$ takes value in $A$. Then, when $\nu(\cdot, B)$ is absolutely continuous with respect to $\mu$, the occupation density (or local time) is given by the Radon-Nikodym derivative
\begin{align}
\notag L(x, B) = \frac{d \nu}{d \mu}(x, B).
\end{align}
\noindent We study the local time via the analytic approach \CR{initiated by} Berman~\cite{Berman-1969}. The idea is to relate properties of $L(x, B)$ with the integrability properties of the Fourier transform of $f$. Recall the following key result~\cite{Dozzi-2003}:
\begin{proposition}\label{prop:dozzi} The function $f$ has an occupation density $L(x, B)$ for $x \in\mathbb{R}, B \in \mathcal{B}([u,U])$ which is square integrable in $x$ for every fixed $B$ iff
\begin{align}
\notag \int_{\mathbb{R}} \left| \int_u^U \exp (i \xi f(t) )dt \right|^2 d \xi < \infty.
\end{align}
Moreover, in this case, the occupation density can be represented as
\begin{align}
\notag L(x, B) = \frac{1}{2 \pi} \int_\R \int_B \exp ( i \xi (x- f(s)) ) d\xi ds.
\end{align}
\end{proposition}
The deterministic function $f(t)$ can be chosen to be the single path of a stochastic process $(X_t)_{t\geq 0}$. To show almost sure existence and square integrability of $L(x,B)$ in that case, it will be enough to show that
\begin{align}
\notag \mathbb{E} \left[  \int_{\mathbb{R}} \left| \int_u^U \exp (i \xi X_t )dt \right|^2 d \xi \right] < \infty,
\end{align}
or equivalently
\begin{align}
\label{eq:bound_fourier_ex}   \int_{\mathbb{R}} \int_u^U \int_u^U \E [ \exp (i \xi (X_s - X_t)\CR{)} ] ds dt d \xi < \infty.
\end{align}
If $(X_t)_{t \geq 0}$ is Gaussian, then one can evaluate $\E[ \exp \CR{(i \xi (X_s - X_t))} ]$ explicitly to establish~\eqref{eq:bound_fourier_ex}. It leads to the well-known Gaussian criterion:
\begin{proposition}[Lemma 19 in~\cite{Dozzi-2003}]  Suppose that $X$ is Gaussian and centered, and satisfies
\begin{align}
\notag \int_{[u,U]^2} \Delta(s,t)^{-1/2} ds dt < \infty, 
\end{align}
where $\Delta(s,t) = \E [ (X_s - X_t)^2]$. Then $(X_t)_{t \geq 0}$ has an occupation density $L = L(x, B, \omega)$ which, for $B$ fixed, is $P-$a.s. square integrable in $x$.
\end{proposition}

In our setting $(X_t)_{t \geq 0} = (Z_t)_{t \geq 0}$  is the Rosenblatt process which is not Gaussian. Nevertheless, a careful \CR{analysis of} $\E [\CR{\exp (i \xi (Z_s - Z_t))} ] $  via~\eqref{eq:rosenblatt_spectral_dec} yields~\eqref{eq:bound_fourier_ex}. This is the approach in~\cite{Shevchenko-2010} where existence of the local time of the Rosenblatt process was first established. In this paper we show a \CR{considerably more involved} bound on the Fourier transform which in turn will yield deeper results   regarding the irregularity properties of the sample paths.  The following is the key result of our paper:

 %Still, using~\eqref{eq:rosenblatt_spectral_dec} and some tools from Harmonic analysis one can show that~\eqref{eq:bound_fourier_ex} holds. If fact, we establish more than~\eqref{eq:bound_fourier_ex}, which in turn will yield deeper results   regarding the irregularity properties of the sample paths.  The following is the key result of our paper 

\begin{proposition}\label{prop:int_fourier} Let $n \in \mathbb{N}$ and  $0 \leq \eta \CR{<} \frac{1-H}{2H}$. Then, for any times $0 \CR{\leq} u < U$, the Rosenblatt process satisfies
\begin{align}
\label{eq:bound_fourier}   \int_{[u, U]^n} \int_{\mathbb{R}^n} \prod_{j =1}^n |\xi_j|^{\eta} \left| \mathbb{E} \exp \left( i \sum_{j = 1}^n \xi_j Z_{t_j} \right) \right| d\xi dt  \leq C^n n^{2nH(1 + \eta)} (U - u)^{(1 - H(1 + \eta) ) n},
\end{align} 
\noindent where the constant $C > 0$ depends only on $H$ and $\eta$. 
\end{proposition} 

%\noindent Proposition~\ref{prop:int_fourier} can be used to show the existence of the local time of $Z_t$, thus recovering the results of~\cite{Shevchenko-2010}. This is done in Section~\ref{sec:moments}. 
Proposition~\ref{prop:int_fourier} can be applied to obtain the following  H\"older conditions on $L(x,B)$. 

\begin{theorem}\label{thm:main}
Let $(Z_t)_{t\geq 0}$ be a Rosenblatt process with $H\in\left(\frac12,1\right)$. The local time \CR{
$
(x,t) \mapsto L(x, [0,t]) 
$} is \CR{almost surely} jointly continuous and has finite moments. For a finite closed interval $I \subset (0, \infty)$, let $L^*(I) = \sup_{x\in\R}L(x,I)$. There exist constants $C_1$ and $C_2$ such that, almost surely,
\begin{align}
\label{eq:main-fixed-s}
\limsup_{r\to 0} \frac{L^*([s-r,s+r])}{r^{1-H}(\log\log r^{-1})^{\cexp}} \leq C_1,
\end{align}
for any $s \in I$ and
\begin{align}
\label{eq:main-sup-s}
\limsup_{r\to 0}\sup_{s\in I} \frac{L^*([s-r,s+r])}{r^{1-H}(\log r^{-1})^{\cexp}} \leq C_2.
\end{align}
\noindent 
\end{theorem}
In particular, the local time $L(x,I)$ is well defined for any fixed  $x$ and interval $I\subset(0,\infty)$.  \CR{Explicit estimates for the moments of the local time are provided  in Theorem \ref{thm:moment0-bound} below.} As a direct corollary we obtain: 

\begin{corollary}\label{cor:fixed-point}
For any finite closed interval $I \subset (0,\infty)$ there exists \CR{constants $C_1$ and $C_2$}, independent of $x$ and $t$, such that\CR{, almost surely,} for every $t\in I$ and every $x\in \R$
\begin{align}
\notag \limsup_{r\to 0} \frac{L(x,[t-r,t+r])}{r^{1-H}(\log\log r^{-1})^{\cexp}} \leq C_1,
\end{align}
\CR{and for every $x\in \R$
\begin{align}
\notag \limsup_{r\to 0}\sup_{t\in I} \frac{L(x,[t-r,t+r])}{r^{1-H}(\log r^{-1})^{\cexp}} \leq C_2.
\end{align}
}
\end{corollary}

Moreover, we get the following for a particular Hausdorff measure:
\begin{corollary}\label{cor:hausdorff}
Let $I \subset (0, \infty)$ be a finite closed interval.  There exists a constant $C$ such that for every $x\in \R$ we have
\begin{align}
\notag \CR{{\mathcal H}_\phi}(Z^{-1}(x) \cap I) \geq CL(x,I), \quad \text{ a.s., } 
\end{align}
where \CR{${\mathcal H}_\phi$} denotes $\phi$-Hausdorff measure with $\phi(r) = r^{1-H}(\log \log r^{-1})^{\cexp}$.
\end{corollary}

Furthermore, we can get a result on the behavior of the trajectories of \CR{$Z$}.
\begin{corollary}\label{cor:non-differentiable}
Let $I \subset (0, \infty)$ be a finite closed interval. There exists a constant $C> 0$ such that for every $s\in I$ we have, almost surely,
\begin{align}
\notag \liminf_{r\to 0} \sup_{s-r<t<s+r} \frac{|Z_t-Z_s|}{r^H (\log \log r^{-1})^{-\cexp}} \geq C,
\end{align}
and
\begin{align}
\notag \liminf_{r\to 0}\inf_{s\in I} \sup_{s-r<t<s+r} \frac{|Z_t-Z_s|}{r^H (\log \log r^{-1})^{-\cexp}} \geq C.
\end{align}
In particular, $Z$ is almost surely nowhere differentiable.
\end{corollary}
In Section~\ref{sec:int_fourier} we establish our main result Proposition~\ref{prop:int_fourier}. \CR{As we are dealing with  a second order Hermite process, we are forced to control from below  singular values  of somewhat unwieldy operators (see Remark \ref{rem:operator} below). For this purpose we need to introduce several technical lemmas exhibiting tools from operator theory and harmonic analysis, including the theory of weighted integrals. Their} proofs are postponed \CR{into} Section~\ref{sec:lemmas_proofs}. Section~\ref{sec:moments} is dedicated to some results regarding the existence and joint continuity of the local time. In particular, bounds on the moments of $L(x, B)$ are obtained. Finally, in Section~\ref{sec:proofs}\CR{,}  Theorem~\ref{thm:main} and Corollaries~\ref{cor:fixed-point}-\ref{cor:non-differentiable} are established.

\section{Integrability of the Fourier transform}\label{sec:int_fourier}

The purpose of this section is to provide a proof of Proposition~\ref{prop:int_fourier}. We first outline the main steps as Lemmas~\ref{lem:part_eval}-\ref{lem:alku} and  then we establish~\eqref{eq:bound_fourier}. The proofs of the lemmas are carried out in Section~\ref{sec:lemmas_proofs}.
 
We use the following normalization for the Fourier transform
\begin{align}
\notag \widehat{f}(\xi) \coloneqq \int_{\mathbb{R}} e^{- i \xi x} f(x) dx, 
\end{align}
\noindent for $\xi \in \mathbb{R}$ and $f \in C_0^{\infty}(\mathbb{R})$. The norm in the weighted space $L^2_G$ is defined as $\|f\|_{L^2_G}^2:=\int_\R |f(x)|^2G(x)dx.$

First, we obtain a representation of the left-hand side of~\eqref{eq:bound_fourier} using the eigenvalues of an integral operator. 
\begin{lemma}\label{lem:part_eval} Let $t \in \R_+^n$, $\xi \in \mathbb{R}^n$, and let $A_{t, \xi} : L^2_G(\mathbb{R}) \to L^2_G(\mathbb{R})$, be the operator given by 
\begin{align}
\notag (A_{t, \xi} f)(x) = \int_{\mathbb{R}}  \sum_{j = 1}^n  \xi_j \frac{e^{i t_j (x - y) } -1 }{i (x - y)} f(y) |y|^{-\CR{H}}dy 
\end{align}
Let $(\lambda_k)_{k \geq 1}$ be \CR{the sequence of the singular values}\footnote{\CR{Note that for a self-adjoint operator $T$  the singular value sequence consists of absolute values of the eigenvalues, repeated according to multiplicity. For later  purposes it is useful to speak of singular values since some of the results we will be using in Section \ref{sec:spectral} are valid only for singular values.}}
 of $A_{t, \xi}$. Then, 
\begin{align}
\notag \left| \mathbb{E} \exp \left( i \sum_{j = 1}^n \xi_j Z_{t_j} \right) \right|  = \prod_{k  \geq 1} \frac{1}{( 1 + 4 \lambda_k\CR{^2})^{1/4}}.
\end{align}
\end{lemma}

\noindent Next, instead of studying the properties of $\lambda_k$, defined as in Lemma~\ref{lem:part_eval}, we introduce an operator that is unitarily equivalent to $A_{t, \xi}$, and obtain estimates for its singular values. 

\begin{lemma}\label{lem:part_new_operator} The operator $A_{t, \xi}$ is unitarily equivalent to 
\begin{align}
\label{eq:b_operator} B_{t, \xi} \coloneqq c(H) K_{H/2} M_g K_{H/2}\; \CR{: \L L^2(\R)\to L^2(\R)}
\end{align}
where $g(x) = \sum_{j = 1}^n \xi_j \chi_{[0, t_j]}(x)$, $M_g$ is the multiplication operator $(M_g f) (x) \coloneqq g(x) f(x)$ and $K_{\alpha}$ is a convolution operator defined \CR{via the Fourier transform}:
\begin{align}
\notag (\widehat{K_{\alpha} f}) (x) \coloneqq |x|^{- \alpha} \widehat{f}(x),
\end{align} 
\noindent for $\alpha \in (-1/2, 1/2)$.
Furthermore, if  $t_0 = 0 < t_1 < \cdots < t_n  \leq 1$, \CR{and we set $\xi_0=0$}, the $n$th singular value $\mu_n$ of $B_{t, \xi}$ $($and then also of $A_{t, \xi})$ satisfies
\begin{align}
\label{eq:bound_svalue} \mu_n ( B_{t, \xi}) \CR{\; \geq \;} C(H) (\max_{1 \leq j \leq n} |\xi_j - \xi_{j-1} | | t_j - t_{j-1}|^H ) \tilde{\mu}_n^2, 
\end{align}
\noindent where $\tilde{\mu}_n\; \sim\; c(H)n^{- H/2}$, and $c(H), C(H) > 0$ are constants that only depends on $H$. 
\end{lemma}

\noindent The next step is to provide bounds on integrals involving expressions like on the right-hand side of~\eqref{eq:bound_svalue}. First, the following result holds. 

\begin{lemma}\label{lem:int_svalue} Define for $ s \in \mathbb{R}$, the function
\begin{align}
\label{eq:gproduct} G(s) \coloneqq \prod_{k = 1}^{\infty} (1 + 4 s^2 \tilde{\mu}_k^4)^{-1/4}.
\end{align}
Then the product~\eqref{eq:gproduct} converges and $G(s) > 0$ for $s > 0$. \CR{Moreover, there is a constants $c_3=C_3(H)$}  such that, for all $\beta \geq 1$, 
\begin{align}
\label{eq:gbound} \int_0^{\infty} s^{\beta -1} G(s) ds \leq  c_3^{\beta H}\Gamma(\beta H),
\end{align}
\noindent where $\Gamma$ is the Gamma function.  
\end{lemma}

\noindent The next technical result gives an expression for an integral of a function similar to the maximum term appearing in~\eqref{eq:bound_svalue}. In particular, let $f_0: \R_+^n \times \R^n \to \R_+$ be given by
\begin{equation}\label{eq:special}
 f_0(t, y) \coloneqq t_1^H |y_1| \vee t_2^H |y_2|  \vee \cdots \vee t_n^H |y_n|.
\end{equation}
\noindent Then the following holds: 
\begin{lemma}\label{lem:alku}Assume that $\gamma_j \in [0, H^{-1} - 1)$ for each $j$ where $H \in (0, 1)$, and write $\gamma_{av} =n^{-1} \sum_{j =1}^n \gamma_j$. Then,
\begin{align}
\notag \int_{S^{n-1}} & \int_{\substack{ t_1 + \cdots + t_n \leq 1 \\ t_1, \ldots, t_n \geq 0}} \prod_{j =1}^n |y_j|^{\gamma_j} (f_0(t, y))^{-n ( 1 + \gamma_{av})} dt \mathcal{H}^{n-1} (dy)  \\
\label{eq:alku} = & \quad \frac{n^{1/2} (1 + \gamma_{av} ) \prod_{j =1}^n \left[ \frac{2}{1 + \gamma_j} \Gamma( 1 - H(1 + \gamma_j) ) \right]}{\Gamma(n (1 - H(1 + \gamma_{av})) + 1)},
\end{align} 
where $\mathcal{H}^{n-1}(dy)$ is \CR{the $(n-1)$-dimensional Hausdorff measure}.
\end{lemma} 

\noindent Now that the main technical steps are outlined, we present the proof of Proposition~\ref{prop:int_fourier}. 

\begin{proof}[Proof of Proposition~\ref{prop:int_fourier}]
Note first that\CR{,} by Lemma~\ref{lem:part_eval}\CR{,}
\begin{eqnarray*}
&&\int_{[u,U]^n}\int_{\R^n}\prod_{j=1}^n|\xi_j|^\eta\left\vert\E \exp\left(i\sum_{j=1}^n \xi_jZ_{t_j}\right)\right\vert d\xi dt \\
&=&\int_{t\in [u,U]^n}\int_{\xi\in\R^n}\prod_{j=1}^n |\xi_j|^{\eta}\prod_{k=1}^\infty \big(1+4\mu_k(A_{t,\xi})^2)\big)^{-1/4}d\xi dt \;=:\; I.
\end{eqnarray*}
We first perform a reduction to the case $u=0$. For that purpose  recall that Lemma \ref{lem:part_new_operator} verifies that the operator $A_{t,\xi}$ is unitarily equivalent to the operator $c(H/2)\sum_{j=1}^n\xi_jK_{H/2} M_{\chi_{[0,t_j]}}K_{H/2}  $. Next, assuming that $t_j\geq u$ for all $j\in\{1,\ldots,n\}$, Lemma \ref{le:piecing} yields that, for every $k \geq 1$,
\begin{eqnarray*}
&&\mu_k\big(\sum_{j=1}^n\xi_jK_{H/2}  M_{\chi_{[u,t_j]}}K_{H/2} \big) = \mu_k\big(U_{{H/2} ,[u,\infty)}\sum_{j=1}^n \xi_jK_{H/2}  M_{\chi_{[0,t_j]}}K_{H/2} \big) \\
&\leq& \co\mu_k(A_{t,\xi}),
\end{eqnarray*}
where we used the Minimax principle (see~\cite[Corollary III.1.2]{Bhatia-1997}) and $\co \coloneqq \|U_{H/2,[u,\infty)}\|$\CR{.} On the other hand, by the translation invariance of the Fourier-multipliers $K_\alpha$ we see that, for every $k \geq 1$,
$$
\mu_k\big(\sum_{j=1}^n \xi_jK_{H/2}  M_{\chi_{[u,t_j]}}K_{H/2} \big)= \mu_k\big(\sum_{j=1}^n\xi_jK_{H/2}  M_{\chi_{[0,t_j-u]}}K_{H/2} \big).
$$
Then, $\mu_k(A_{t,\xi})\geq \co^{-1}\mu_k(A_{t-ue,\xi})$, where  $ e = (1,\ldots ,1)$, and hence a change of variables in the integral $I$ yields
\begin{equation}\label{eq:draw}
I\leq \int _{t\in [U-u]^n}\int_{\xi\in\R^n}\prod_{j=1}^n |\xi_j|^{\eta}\prod_{k=1}^\infty \big(1+4A^{-1}\mu_k(A_{t,\xi})^2)\big)^{-1/4}d\xi dt :=I'.
\end{equation}
Above, $I'$ is our integral $I$ reduced to the case $u=0$ up to a constant $\co$ in the integrand.

We then assume that $u=0$ and consider the integral $I'$. By symmetry, $I'$ equals $n!$ times  the same integral restricted to the ordered set $\{0\leq t_1< \cdots <  t_n \leq U-u\}.$ Again, $A_{t,\xi}$ is unitarily equivalent to 
$K_{H/2} M_gK_{H/2},$  where 
$$
g=\sum_{j=1}^n \xi_j\chi_{[0,t_j]}=\sum_{j=1}^n\xi'_j\chi_{I_j},
$$
where $I_j=[t_{j-1},t_j]$ with $t_0:=0$ and $\xi'_j \coloneqq \sum_{\ell=j}^n\xi_j$, $j=1,\ldots ,n$. For notational purposes we set $\xi'_{n+1} = 0$ We also   make the change of  variables $t'_j=t_j-t_{j-1}$, $j=1,\ldots ,n$.   Both changes of variables have Jacobian equal to 1. According to~\eqref{eq:bound_svalue} of Lemma~\ref{lem:part_new_operator}  and Lemma~\ref{lem:int_svalue}: 
$$
I'\leq n!\int _{\substack{t'_1+\ldots +t'_n\leq U\\ t'_1,\ldots, t'_n >0}}\int_{\xi' \in\R^n}\prod_{j=1}^n |\xi'_j-\xi'_{j+1}|^{\eta}G(c'f_0(t',\xi'))dt'd\xi'
$$
where $c'>0$ (which also incorporates the constant $\co$) and  $f_0$ \CR{is as in \eqref{eq:special}}.

Next, note that for every $\eta\CR{\geq0}$, $|\xi'_j-\xi'_{j+1}|^{\eta}\leq 2^{(\eta-1)\vee 0}(|\xi'_j|^\eta+|\xi'_{j+1}|^\eta)$ and thus
$$
\prod_{j=1}^n |\xi'_j-\xi'_{j+1}|^{\eta}\leq c''^n \sum \prod |\xi'_j|^{\gamma_j},
$$ 
where $c'' \coloneqq 2^{(\eta-1)\vee 0}$ and the exponents in each term of the sum satisfy $\gamma_j\in \{0,\eta,2\eta\}$, and 
 $\gammav :=n^{-1}\sum_{j=1}^n \gamma_j=\eta.$ The number of summands is $2^{n-1}$, since $\xi'_{n+1}=0$. For any  fixed exponent sequence 
$(\gamma_1,\ldots ,\gamma_n)$ we switch to polar coordinates $|\xi'|=r'$, $\xi'/r=w'$:
\begin{eqnarray*}
&&\int _{\substack{t'_1+\ldots +t'_n\leq U\\ t'_1,\ldots, t'_n >0}}\int_{\xi' \in\R^n}\prod_{j=1}^n |\xi'_j|^{\gamma_j}G(c'f_0(t',\xi'))dt'd\xi'\\
&=&\int _{\substack{t'_1+\ldots +t'_n\leq U\\ t'_1,\ldots, t'_n >0}}\int_{|w'|=1}\int_0^\infty r'^{n-1}r'^{n\eta}\prod_{j=1}^n |w'_j|^{\gamma_j}G(c'r'f_0(t',w'))dr'\mathcal{H}^{n-1}(dw') dt'\\
&=&(c')^{-n(1+\eta)}\Big(\int_0^\infty R^{n(1+\eta)-1}G(R)dR\Big)\times\\
&&\phantom{kukkuukukkuu}\times\int _{\substack{t'_1+\ldots +t'_n\leq U\\ t'_1,\ldots, t'_n >0}}\int_{|w'|=1}\prod_{j=1}^n |w'_j|^{\gamma_j}(f_0(t',w'))^{-n(1+\eta)}\mathcal{H}^{n-1}(dw') dt',
\end{eqnarray*}
where  $R \coloneqq  c' r' f_0(t', w')$.  Apply  Lemma \ref{lem:int_svalue} to estimate the $R$-integral and set $t'= Uv$.  By the $H$-homogeneity of $f_0$ in the $t$-variable the previous integral is upper bounded by
\begin{eqnarray*}
&& c_3^{ n(1+\eta)H} (c')^{-n(1+\eta)}\Gamma(n(1+\eta)H)\times\\
&&\phantom{kukkuu}\times U^{(1-H(1+\eta))n}\int _{\substack{v_1+\ldots +v_n\leq 1\\ v_1,\ldots, v_n >0}}\int_{|w'|=1}\prod_{j=1}^n |w'_j|^{\gamma_j}(f_0(v,w'))^{-n(1+\eta)}\mathcal{H}^{n-1}(dw)dv
\end{eqnarray*}
for some constant $c_3=C_3(H)>0$. Next, by Lemma~\ref{lem:alku} the above is further bounded by
\begin{align*}
  &C(H,n) U^{(1-H(1+\eta))n}\Gamma(n(1+\eta)H)  \frac{n^{1/2}(1+\eta)\prod_{j=1}^n\Big[\frac{2}{1+\gamma_j}\Gamma(1-H(1+\gamma_j))\Big]}{(c')^{n(1+\eta)} \Gamma(n(1-H(1+\eta))+1)}\\
 \leq\;\; & C(H,n)  U^{(1-H(1+\eta))n}\Big(\frac{2}{c'^{1+\eta}}\Gamma (1-H (1+2\eta))\Big)^n\frac{\Gamma(nH(1+\eta)+1)}{n^{1/2}\Gamma(n(1-H(1+\eta))+1)},
\end{align*}
where $C(H,n) :=  (c_3/c')^{- n(1+\eta)H} $ and we used the fact that the Gamma function is decreasing on $(0,1)$ and $\gamma_j \leq 2 \eta$. Summing over the $2^{n-1}$ different  exponent sequences 
$(\gamma_1,\ldots ,\gamma_n)$  and  recalling the $n!$ factor  introduced in the beginning of the proof yields
\begin{align}
\notag I \leq  & C(H) U^{(1-H(1+\eta))n}\Big(\frac{4}{c'^{1+\eta}}\Gamma (1-H (1+2\eta))\Big)^n\frac{\Gamma(n+1)\Gamma(nH(1+\eta)+1)}{\Gamma(n(1-H(1+\eta))+1)} \\
\notag = & U^{(1-H(1+\eta))n} C^n \frac{\Gamma(n+1)\Gamma(nH(1+\eta)+1)}{\Gamma(n(1-H(1+\eta))+1)},
\end{align}
where $C > 0$ depends only on $H$ and $\eta$. 
 Finally, an application of Stirling's formula establishes~\eqref{eq:bound_fourier}:
\begin{align}
\notag I \leq    C^n n^{2nH(1 + \eta)} \CR{U}^{(1 - H(1 + \eta) ) n},
\end{align} 
for a different $C > 0$ depending only on $H$ and $\eta$ \CR{(recall that we reduced to the case $u=0$)}. 
\end{proof}

Proposition~\ref{prop:int_fourier} establishes integrability properties of the Fourier transform of the local time and leads to  good moment estimates for the local time in next section. In turn, Section \ref{sec:proofs} uses the moment estimates to deduce several important results regarding the asymptotic behavior of the local time. 

\section{Joint continuity of the local times \CR{and moment estimates}}\label{sec:moments}

\CR{In the  present  section we apply Proposition~\ref{prop:int_fourier} to produce moment estimates for the local time, that are of some independent interest.}

Let $t > 0$ and $x \in\R$. We recall from  \cite{Shevchenko-2010} that the local time  $L(x, t) \coloneqq L(x, [0,t])$ for the Rosenblatt process $Z$ exists and admits the representation\footnote{\CR{A priori, after  \cite{Shevchenko-2010}, the occupation density is defined only as an $L^2$-density, and hence they are not necessarily well-defined for fixed $x$. Thus some of our computations below might seem unfounded. However, one may use Proposition \ref{prop:int_fourier} to first prove uniform regularity bounds for suitable mollifications, which justifies \eqref{eq:local_rep} for any fixed $x$ and our computations later on.}}
\begin{align}
\label{eq:local_rep} L(x, t) = \frac{1}{2 \pi} \int_\R \int_0^t e^{i \xi (x - Z_s) } ds d\xi.
\end{align}

Our next step is to show that $L(x, t)$ is H\"older-continuous both in time and space, and also to establish bounds on its moments.

\begin{theorem}\label{thm:moment0-bound}
For every $0 \leq s < t$ and $x \in \R$, 
\begin{equation}\label{eq:mom3}
\E |L(x,t)-L(x,s)|^n \leq c^n n^{n\cexpf}|t-s|^{(1-H)n}.
\end{equation}
Moreover, for any $0\leq \gamma<\frac{1-H}{2H}$ and \CR{$y \in \R$}, we have
\begin{equation}\label{eq:mom4}
\left| \E (L(x+y,[s,t])-L(x,[s,t]))^{n} \right| \leq c^n n^{n \cexps}|t-s|^{(1-H-\gamma H)n}|y|^{\gamma n}. 
\end{equation}
In both inequalities the constant $c$ depends only on $\gamma$ and $H$.
%is independent of $n,t,s,x,$ and $y$.
\end{theorem}
\begin{proof}
First, by~\eqref{eq:local_rep}, 
\begin{align}
\notag & \E \big(L (x+y,[s,t])-L(x,[s,t])\big)^{n} \\
\notag  = & (2 \pi)^{-n} \int_{\R^n} \int_{[s, t]^n} \left( \prod_{j =1}^n ( \exp ( i \xi_j (x+ y) ) - \exp( i \xi_j x) ) \right)\E \exp\left( - i \sum_{j = 1}^n \xi_j Z_{v_j} \right) dv d\xi  
\end{align}
Next, since $\gamma \in [0,1 )$ (recall that $H > 1/2$),  we have
\begin{align}
\notag & \prod_{j =1}^n | \exp ( i \xi_j y ) - 1 | \leq 2^n \prod_{j =1}^n (|y| |\xi_j| \wedge 1) \leq 2^n \prod_{j =1}^n (( |y| |\xi_j| )^{\gamma} \wedge 1) \leq 2^n |y|^{\gamma n} \prod_{j =1}^n |\xi_j|^\gamma,
\end{align}
where we have used that  $|e^{ix} - 1| \leq |x| \wedge 2 \leq 2 ( |x| \wedge 1)$, for all $x$. 
 Therefore, 
\begin{align}
\notag & \left|  \E \big(L (x+y,[s,t])-L(x,[s,t])\big)^{n} \right| \\ 
\notag \leq  & \pi^{-n} |y|^{\gamma n} \int_{\R^n} \int_{[s, t]^n} \prod_{j =1}^n |\xi_j|^\gamma  \left| \E \exp\left( - i \sum_{j = 1}^n \xi_j Z_{v_j} \right)\right| dv d\xi.  
\end{align} 
Now, Proposition~\ref{prop:int_fourier} with $\eta = \gamma$ yields: 
\begin{align}
\notag & \left| \E \big(L (x+y,[s,t])-L(x,[s,t])\big)^{n} \right| \leq C^n |y|^{\gamma n} (t - s)^{1 - H(1 + \gamma) n} n^{2 n H( 1 + \gamma) },
\end{align}  
where $C > 0$ is a function of $H$ and $\gamma$ and~\eqref{eq:mom4} is established.

Similarly, by~\eqref{eq:local_rep}, using $L(x, s) \leq L(x, t)$ for $0 \leq  s < t$, 
\begin{eqnarray*}
 &  & \E |L(x,t)-L(x,s)|^n \\
& = & \left((2\pi)^{-n}\int_{[s,t]^n}\int_{\R^n}\exp\left(ix\sum_{j=1}^n \xi_j\right)\E\exp\left(-i\sum_{j=1}^n \xi_j Z_{u_j}\right)d\xi du\right) \\
&\leq& (2\pi)^{-n}\int_{[s,t]^n}\int_{\R^n}\left\vert\E\exp\left(-i\sum_{j=1}^n \xi_j Z_{u_j}\right)\right\vert d\xi du\\
&\leq& \Big(\frac{C}{2\pi}\Big)^nn^{2nH}(t-s)^{(1-H)n},
\end{eqnarray*}
where the last inequality follows from Proposition~\ref{prop:int_fourier} with $\eta = 0$, and  $C > 0$ is a function of $H$ and $\gamma$.
\end{proof}

 As an immediate consequence of the above moment bounds and Kolmogorov criterion (see e.g.~\cite[Theorem 3.23]{ Kallenberg-2002}) we obtain:
 \begin{corollary}\label{cor:holder}
 Almost surely, the local time $L(x,t)$ is jointly H\"older continuous in $t$ and $x$.
 \end{corollary}

The next theorem is a modification of Theorem~\ref{thm:moment0-bound}, where one shifts the process in the $x$-direction by the value $Z_a$, where $a$ is a fixed point.

\begin{theorem}
\label{thm:moment}
\CR{Let $s<t$ and let $a > 0$ satisfy $a \leq s$ or $a\geq t$}. Then,
\begin{equation}\label{eq:mom3-shift}
\E |L(x+Z_a,t)-L(x+Z_a,s)|^n \leq c^n (n)^{n\cexpf}|t-s|^{(1-H)n}.
\end{equation}
Moreover, for any $0\leq \gamma<(H^{-1}-1)/2$,
\begin{equation}\label{eq:mom4-shift}
\left| \E (L(x+y+Z_a,[s,t])-L(x+Z_a,[s,t]))^{n} \right| \leq c^n (n)^{n\cexps}|t-s|^{(1-H-\gamma H)n}|y|^{\gamma n}.
\end{equation}
In both cases the constant $c > 0$ depends only on $\gamma$ and $H$.
%is independent of $n,a,t,s,x$ and$y$.
\end{theorem}
\begin{proof}
Let $Y_t = Z_t - Z_a$. The occupation measure of $Y$ is just the occupation measure of $Z_t$ translated by the (random) constant $Z_a$. Since the occupation measure of $Z_t$ has a continuous density, the occupation measure of $Y_t$ has also a continuous density given by
$L_Y(t,x) = L_Z(t,x+Z_a)$. Thus, in order to prove the claim, it suffices to show the estimates for $L_Y(t,x)$. For the first claim, we then proceed as before, noting that, again, $L_Y(x, s) \leq L_Y(x, t)$,
\begin{eqnarray*}
&&\E |L_Y(x,t)-L_Y(x,s)|^n \\&= &(2\pi)^{-n}\int_{[s,t]^n}\int_{\R^n}\exp\left(ix\sum_{j=1}^n \xi_j\right)\E\exp\left(i\sum_{j=1}^n \xi_j Y_{u_j}\right)d\xi du\\
&\leq& (2\pi)^{-n} \int_{[s,t]^n}\int_{\R^n}\left\vert\E\exp\left(i\sum_{j=1}^n y_j Y_{u_j}\right)\right\vert dydu.
\end{eqnarray*}
\CR{Let first $a\leq s$.} By stationarity of the increments, we have 
\begin{equation*}
\begin{split}
&\E\exp\left(i\sum_{j=1}^n y_j Y_{u_j}\right)
\;=\; \E\exp\left(i\sum_{j=1}^n y_j (Z_{u_j}-Z_a)\right)
\;=\;\E\exp\left(i\sum_{j=1}^n y_j Z_{u_j-a}\right).
\end{split}
\end{equation*}
Thus change of variable $v_j=u_j-a$ gives 
$$
\int_{[s,t]^n}\int_{\R^n}\left\vert\E\exp\left(i\sum_{j=1}^n y_j Y_{u_j}\right)\right\vert dydu = \int_{[s-a,t-a]^n}\int_{\R^n}\left\vert\E\exp\left(i\sum_{j=1}^n y_j Z_{v_j}\right)\right\vert dydv
$$
from which the claim follows by Proposition \ref{prop:int_fourier} just as in the proof of Theorem \ref{thm:moment0-bound}  with $\eta = 0$, and  $C > 0$ is a function of $H$ and $\gamma$. \CR{Similarly, for $a\geq t$ stationarity of increments imply
\begin{equation*}
\begin{split}
&\E\exp\left(i\sum_{j=1}^n y_j Y_{u_j}\right)
\;=\; \E\exp\left(-i\sum_{j=1}^n y_j (Z_a-Z_{u_j})\right)
\;=\;\E\exp\left(-i\sum_{j=1}^n y_j Z_{a-u_j}\right).
\end{split}
\end{equation*}
This can be treated exactly the same way by using first change of variable $v_j = a-u_j$. This concludes the proof.}
\end{proof}
The moment bounds obtained above translate  into the following tail estimates. Their proof is a standard application of Chebychev's inequality, and hence we omit the proof.
\begin{corollary}\label{cor:tail1} {\rm (i)}\quad  For any finite closed interval $I\subset (0, \infty)$, 
\begin{equation}
\label{eq:tail-L0}
\mathbb{P}(L(x,I)\geq |I|^{1-H}u^{\cexpf}) \leq C_1\exp(-c_1u)
\end{equation}
and
\begin{equation}
\label{eq:tail-inc0}
\mathbb{P}(|L(x,I)-L(y,I)|\geq  |I|^{1-H-\gamma H}|x-y|^{\gamma}u^{\cexps}) \leq C_2\exp(-c_2u).
\end{equation}

\medskip

\noindent {\rm (ii)}\quad \CR{For $I = [a,a+r]$ or $I=[a-r,a]$, we have}
\begin{equation}
\label{eq:tail-L}
\mathbb{P}(L(x+Z_a,I)\geq r^{1-H}u^{\cexpf}) \leq C_1\exp(-c_1u)
\end{equation}
and
\begin{equation}
\label{eq:tail-inc}
\mathbb{P}(|L(x+Z_a,I)-L(y+Z_a,I)|\geq  r^{1-H-\gamma H}|x-y|^{\gamma}u^{\cexps}) \leq C_2\exp(-c_2u).
\end{equation}
\end{corollary}

\section{Proofs of main theorems}\label{sec:proofs}

\CR{In this section we present proofs to our main results, namely Theorem~\ref{thm:main} and Corollaries~\ref{cor:fixed-point}-\ref{cor:non-differentiable}.}
We start with two auxiliary lemmas.
\begin{lemma}\label{le:eta}
There exists $\eta>0$ such that 
$$
\E e^{\eta|Z_1|} < \infty.
$$
\end{lemma}
\begin{proof}
This follows by observing that by Lemma~\ref{lem:part_eval} the characteristic function of $Z_1$ has a bounded analytic extension to a strip $\{|{\rm Im}\, \theta |<\delta_0\}$ for some $\delta_0>0$.
\end{proof}
\begin{proposition}
\label{prop:sup-tail}
Let $(Z_t)_{t \geq 0}$ be the Rosenblatt process and set \CR{$I = [s-h,s+h]$}, where $h\leq 1$ and $s > 0$. Then,
$$
\mathbb{P}(\sup_{t\in I}|Z_t-Z_s|\geq u) \leq C\exp\left(-\frac{u}{c_1h^{H}}\right),
$$
where $c_1$ and $C$ are constants that depend only on $H$. 
\end{proposition}
\begin{proof}
By stationarity of increments and self-similarity we have the distributional equality
$
\CR{(|Z_r-Z_v|)_{r\in I} \sim  ( |h|^{H}|Z_r|)_{r\in [0,1].}}
$
Hence it is enough to consider case $s=0$ and $h=1$. At this point we recall the Garsia-Rodemich-Rumsey inequality~\cite{Garsia-Rodemich-Rumsey-1970}.
\begin{proposition} Let $\Psi(u)$ be a non-negative even function on $(- \infty, \infty)$ and $p(u)$ be a non-negative even function on $[-1,1]$. Assume both $p(u)$ and $\Psi(u)$ are non decreasing for $u \geq 0$. Let $f(x)$ be continuous on $[0, 1]$ and suppose that
\begin{align}
\notag \int_0^1 \int_0^1 \Psi\left( \frac{f(x) - f(y) }{p(x - y)} \right)dx dy \leq B < \infty.
\end{align}
Then, for all $s, t \in [0, 1]$, 
\begin{align}
\notag |f(\CR{s}) - f(t) | \leq 8 \int_0^{|s - t|} \Psi^{-1} \left( \frac{4B}{u^2} \right) dp(u).
\end{align}
\end{proposition}
In particular\footnote{\CR{Actually, this special case can also be obtained from a standard trace theorem for Besov spaces.}}, if $\Psi(u) = |u|^p$ and $p(u) = |u|^{\alpha + 1/p}$ where $\alpha \geq 1/p$ and $p \geq 1$, then for any continuous $f$ and $t \in [0,1]$, 
$$
|f(t)-f(0)|^p \leq C_{\alpha,p}t^{\alpha p -1}\int_{[0,1]^2}|f(r)-f(v)|^p |r-v|^{-\alpha p - 1}dr dv.
$$
Here the constant $C_{\alpha,p}$ is given by
$
C_{\alpha,p} = \CR{4\cdot 8^p\big(\alpha + p^{-1}\big)^{p}\big(\alpha -p^{-1}\big)^{-p}}.
$
Thus, for fixed $\alpha$ and large enough $p$, we have
$
C_{\alpha,p} \leq \tilde{C}^p,
$
where $\tilde{C}$ depends on the chosen $\alpha$.  We apply this
 to $f(t) = Z_t$  and choose  $\alpha = \frac{H}{2}$, in order to obtain for $p>4/H$:
$$
\sup_{t\in [0,1]}|Z_t|^p \leq C'^p \int_{[0,1]^2}|Z_r-Z_v|^p |r-v|^{-Hp/2 - 1}dr dv.
$$
By stationarity of increments and self-similarity we have $\E|Z_r-Z_v|^p = |r-v|^{Hp}\E|Z_1|^p,$
which leads to 
\begin{equation*}
\begin{split}
\E\sup_{t\in [0,1]}|Z_t|^p &\leq\; 
 C'^p \E|Z_1|^p\int_{[0,1]^2} |r-v|^{Hp/2 - 1}dr dv 
\; \leq \; c_2C'^p\E|Z_1|^p\\
&\leq\; C^p\E|Z_1|^p.
\end{split}
\end{equation*}
By Lemma \ref{le:eta} we may choose  $\eta>0 $ such that 
$
\E e^{\eta|Z_1|} <\infty.
$
Then we have 
\begin{equation*}
\begin{split}
\E e^{\eta C^{-1} \sup_{t\in [0,1]}|Z_t|} &
= \sum_{k=0}^\infty \frac{(\eta C^{-1})^k\E\sup_{t\in [0,1]}|Z_t|^k}{k!} \;
\leq \;\sum_{k=0}^\infty \frac{\eta^k\E|Z_1|^k}{k!}\\
&= \E \sum_{k=0}^\infty \frac{\eta^k|Z_1|^k}{k!} \;
= \; \E e^{\eta|Z_1|} <\infty
\end{split}
\end{equation*}
The claim follows easily from this by Chebyshev's inequality.
\end{proof}

We are now in position to prove Theorem \ref{thm:main}. After we have the moment and tail estimates from Section \ref{sec:moments} at our disposal,  the remaining ideas of the proof follow closely those of  \cite[Theorem 4.3]{xiao-et-al}. The main difference is that in our case, we do not have Gaussian structures at our disposal leading to some modifications.
\begin{proof}[Proof of Theorem \ref{thm:main}]
 We divide the proof into five steps. In the first four steps we prove \eqref{eq:main-fixed-s}. The proof of \eqref{eq:main-sup-s} will then be established in step 5. 

Throughout steps 1-4 we denote $g(r) = r^{1-H}(\log \log r^{-1})^{\cexp}$, where $r<e$, and $C^{\CR{+}}_n := [s,s+2^{-n}]$\CR{, $C_n^- := [s-2^{-n},s]$}. It  actually suffices to prove 
$$
\limsup_{n\to\infty} \frac{L^*(C^{\CR{\pm}}_n)}{g(2^{-n})} \leq C
$$ 
almost surely. \CR{Moreover, it suffices to consider only the interval $C_n^+$, as $C_n^-$ can be treated by exactly the same arguments and by considering the interval $[a-r,a]$ in Corollary \ref{cor:tail1}.} Throughout the proof, we denote by $c_i,i=1,2,\ldots$ generic constants that varies throughout the proof. \CR{We also write simply $C_n = [s,s+2^{-n}]$ instead of $C_n^+$.}

\smallskip

\noindent \textbf{Step 1:}
Set $u= 2c_12^{-nH}\log n$, where $c_1$ is given in Proposition \ref{prop:sup-tail}. Then Proposition \ref{prop:sup-tail} gives
$$
\mathbb{P}\left(\sup_{t\in C_n}|Z_t-Z_s|\geq 2c_12^{-nH}\log n\right) \leq c_2\exp(-2\log n) = c_2n^{-2}.
$$
Hence Borel-Cantelli lemma implies that there exists $n_1 = n_1(\omega)$ such that for $n\geq n_1$ we have
$$
\sup_{t\in C_n}|Z_t-Z_s|\leq 2c_12^{-nH}\log n.
$$
\textbf{Step 2:}
Set $\theta_n = 2^{-nH}(\log \log 2^n)^{-\cexpt}$ and define
$$
G_n =\{x \in \R : |x|\leq 2c_12^{-nH}\log n, x= \theta_n p, \text{ for some }p \in \Z\}.
$$
Then 
$$
\# G_n \leq c_3 (\log n)^{1+\cexpt}
$$
and \eqref{eq:tail-L} implies
$$
\mathbb{P}(\max_{x\in G_n} L(x+Z_s,C_n) \geq c_4 g(2^{-n})) \leq c_5  (\log n)^{1+\cexpt}n^{-c_6b_1}
$$
which is summable by choosing $c_4$ large enough which in turn forces $c_6$ large. Thus Borel-Cantelli gives for large enough $n\geq n_2(\omega)\geq n_1(\omega)$ that
\begin{equation}
\label{eq:max-L-bounded}
\max_{x\in G_n} L(x+Z_s,C_n) \leq c_4 g(2^{-n})).
\end{equation}
\textbf{Step 3:}
Given integers $n,k\geq 1$ and $x\in G_n$ we set
$$
F(n,k,x) = \{ y = x + \theta_n \sum_{j=1}^k \varepsilon_j2^{-j} \;:\; \varepsilon_j \in \{0,1\}, 1\leq j \leq k\}.
$$
Pair of points $y_1,y_2 \in F(n,k,x)$ is said to be linked if $y_2-y_1 = \theta_n \varepsilon 2^{-k}$ for $\varepsilon \in\{0,1\}$. Next fix $0<\gamma< \frac{\frac{1}{H}-1}{2}$ and choose positive $\delta$ such that 
$\delta \cexps < \gamma$. Set
$$
B_n = \bigcup_{x\in G_n}\bigcup_{k=1}^\infty \bigcup_{y_1,y_2} \{|L(y_1+Z_s,C_n)-L(y_2+Z_s,C_n)| \geq 2^{-n(1-H-\gamma H)}|y_1-y_2|^{\gamma}(c_72^{\delta k}\log n)^{\cexps}\}
$$
where $\cup_{y_1,y_2}$ is the union over all linked pairs $y_1,y_2\in F(n,k,x)$. Now \eqref{eq:tail-inc} with $u = c_72^{\delta k}\log n$ implies
$$
\mathbb{P}(B_n) \leq c_3(\log n)^{1+ \cexpt}\sum_{k=1}^\infty 4^{k}\exp\left(-c_82^{\delta k}\log n\right).
$$
Here we have used the fact $\# G_n \leq C(\log n)^{1+\cexpt}$ and that for given $k$ there exists less than $4^k$ linked pairs $y_1,y_2$. Now, again by choosing $c_7$ large enough which makes $c_8$ large, we get 
$$
\sum_{n=2}^\infty(\log n)^{1+ \cexpt}\sum_{k=1}^\infty 4^{k}\exp\left(-c_82^{\delta k}\log n\right) < \infty.
$$
This further implies, again by Borel-Cantelli lemma, that $B_n$ occurs only finitely many times.

\textbf{Step 4:}
Let $n$ be a fixed and assume that $y \in \R$ satisfies $|y| \leq 2c_12^{-nH}\log n$. Then we may represent $y$ as $y = \lim_{k\to\infty}  y_k$ with
$$
y_k = x + \theta_n \sum_{j=1}^k \varepsilon_j 2^{-j},
$$
where $y_0= x \in G_n$ and $\varepsilon_j \in \{0,1\}$. 
On the event $B_n^c$ we have 
\begin{equation*}
\begin{split}
|L(x+Z_s,C_n)-L(y+Z_s,C_n)| & \leq \sum_{k=1}^\infty |L(y_k+Z_s,C_n)-L(y_{k-1}+Z_s,C_n)| \\
&\leq \sum_{k=1}^\infty 2^{-n(1-H-\gamma H)}|y_k-y_{k-1}|^{\gamma}(c_72^{\delta k}\log n)^{\cexps} \\
&\leq \sum_{k=1}^\infty 2^{-n(1-H-\gamma H)}\theta_n^{\gamma}2^{-k\gamma}(c_72^{\delta k}\log n)^{\cexps} \\
& \leq c_9 2^{-n(1-H)} \sum_{k=1}^\infty (\log \log 2^n)^{-\gamma\cexpt}2^{-k\gamma}(c_72^{\delta k}\log n)^{\cexps} \\
& \leq c_{10} 2^{-n(1-H)} (\log \log 2^n)^{\cexp} \sum_{k=1}^\infty 2^{(\delta \cexps - \gamma)k} \\
& \leq c_{11}g(2^{-n}),
\end{split}
\end{equation*}
where the last inequality follows from 
$\delta \cexps < \gamma$. Combining this with \eqref{eq:max-L-bounded} then yields
$$
\sup_{|x|\leq 2c_12^{-nH} \log n} L(x+Z_s,C_n) \leq c_{12} g(2^{-n})
$$
or in other words,
$$
\sup_{|x-Z_s|\leq 2c_12^{-nH}\log n} L(x,C_n) \leq c_{12}g(2^{-n}).
$$
Claim \eqref{eq:main-fixed-s} now follows from Step 1 and the fact $L^*(C_n) = \sup\{L(x,C_n):x\in \overline{Z(C_n)}\}$. 

\textbf{Step 5:}
It remains to prove \eqref{eq:main-sup-s}. However, for this the arguments are similar to above and to the Gaussian case (for detailed proof in the case of the fractional Brownian motion, we refer to \cite{xiao-ptrf}). Thus we present only the key arguments here.

We choose $\tilde{\theta}_n = 2^{-nH} (\log 2^n)^{-\cexpt}$ and
$$
\tilde{G}_n =\{x \in \R : |x|\leq n, x= \tilde{\theta}_n p, \text{ for some }p \in \Z\}.
$$
If $\mathcal{D}_n$ is the dyadic partition of the interval $I$, then the arguments of Step 2 together with \eqref{eq:tail-L} gives that, for $n\geq n_1(\omega)$ and a suitable constant $c_{13}$, we have
\begin{equation}
\label{eq:loctime-bounded}
L(x,B) \leq c_{13}2^{-n(1-H)}(\log 2^n)^{\cexp}
\end{equation}
for all $B \in \mathcal{D}_n$ and $x\in \tilde{G}_n$. Similarly, in Step 3 and Step 4 we replace $F(n,k,x)$ with $\tilde{F}(n,k,x)$ where $G_n$ and $\theta_n$ are replaced with $\tilde{G}_n$ and $\tilde{\theta}_n$, and instead of $B_n$ we consider the event
\begin{equation*}
\begin{split}
\tilde{B}_n &= \left\{|L(y_1,B)-L(y_2,B)| \geq 2^{-n(1-H-\gamma H)}|y_1-y_2|^\gamma (c_{14}k \log 2^n)^{\cexps} :\right. \\
&\left. \text{for some }B\in \mathcal{D}_n, y_1,y_2 \in \tilde{F}(n,k,x) \text{ linked} \right\}
\end{split}
\end{equation*}
for suitably chosen $\gamma$ and constant $c_{14}$.
As in Step 3, then \eqref{eq:tail-inc} gives that $\tilde{B}_n$ occurs only finitely many times. In the complement $\tilde{B}_n^c$ we can apply \eqref{eq:loctime-bounded} and proceed as in Step 4 to conclude that 
for all $B \in \mathcal{D}_n$ 
$$
L(x,B) \leq c_{15}2^{-n(1-H)}(\log 2^n)^{\cexp}.
$$
This completes Step 5 and thus the whole proof. 
\end{proof}
Proofs of corollaries \ref{cor:fixed-point} to \ref{cor:non-differentiable} follows essentially from Theorem \ref{thm:main} and the arguments presented in \cite{xiao-et-al}. Thus we simply state the key arguments and leave the details to the reader.
\begin{proof}[Proof of Corollary \ref{cor:fixed-point}]
This follows directly from Theorem \ref{thm:main} and the fact that
$$
L(x,[t-r,t+r]) \leq L^*([t-r,t+r]).
$$
\end{proof}
\begin{proof}[Proof of Corollary \ref{cor:hausdorff}]
As the proof \CR{of~\cite[Theorem 4.1]{xiao-ptrf}}, the claim follows from Corollary~\ref{cor:fixed-point} and the upper density theorem of~\cite{roger-taylor}. We omit the details.
\end{proof}
\begin{proof}[Proof of Corollary \ref{cor:non-differentiable}]
This claim follows again from Theorem \ref{thm:main} by applying exactly the same arguments as in the proof of Theorem 4.5 of~\cite[Theorem 4.5]{xiao-et-al}. We omit the details. 
\end{proof}

\section{Technical results}\label{sec:lemmas_proofs}

\subsection{Spectral estimates}\label{sec:spectral}

\begin{proof}[Proof of Lemma~\ref{lem:part_eval}] Note that by~\eqref{eq:rosen_rep_kernel} and~\eqref{eq:rosen_kernel}, 
\begin{align}
\notag \sum_{j  =1}^n \xi_j Z_{t_j} = \int_{\mathbb{R}^2} H_t(x, y) Z_G(dx) Z_G(dy), 
\end{align} 
\noindent where the integral is taken over $x \neq \pm y$, $t = (t_1, \ldots, t_n) \in \R_+^n$ and
\begin{align}
\notag H_t(x, y) = \sum_{j = 1}^n  \xi_j \frac{e^{i t_j (x + y) } -1 }{i (x + y)},
\end{align} 
\noindent is a Hilbert-Schmidt kernel. Then the operator $A_{t, \xi}$ satisfies 
\begin{align}
\notag (A_{t, \xi} f)(x) = \int_{\mathbb{R}} H_t(x, -y) f(y) G(dy) = \int_{\mathbb{R}} \sum_{j = 1}^n  \xi_j \frac{e^{i t_j (x - y) } -1}{i (x - y)} f(y) |y|^{-H} dy.
\end{align}

\noindent\CR{Similarly} to~\eqref{eq:rosenblatt_spectral_dec}, for all $\xi \in \mathbb{R}^n$, 
\begin{align}
\notag \sum_{j  =1}^n \xi_j Z_{t_j} \overset{d}{=} \sum_{k =1}^{\infty} \lambda_k (X_k^2 -1),
\end{align} 
\noindent where $(X_k)_{k \geq 1}$ is a sequence of independent Gaussian random variables and $(\lambda_k)_{k \geq 1}$ are the eigenvalues of the operator $A_{t, \xi}$. Then, the characteristic function (evaluated at $1$) of  $\sum_{j  =1}^n \xi_j Z_{t_j} $ is 
\begin{align}
\notag \mathbb{E} \exp \left( i \sum_{j = 1}^n \xi_j Z_{t_j} \right) = \prod_{k \geq 1} \frac{e^{- i \lambda_k} }{ \sqrt{1 - 2i \lambda_k}},
\end{align} 
\noindent where we have used the expression for the characteristic function of a $\chi^2$ distribution and independence. Note, that the product converges since $\sum \lambda_k^2 < \infty$. Furthermore, 
\begin{align}
\notag \left| \mathbb{E} \exp \left( i \sum_{j = 1}^n \xi_j Z_{t_j} \right) \right|  = \prod_{k  \geq 1} \frac{1}{( 1 + 4 \lambda_k\CR{^2})^{1/4}},
\end{align}
\noindent as desired.

\end{proof}

\noindent Before we present the proof of Lemma~\ref{lem:part_new_operator} we will establish some properties of the convolution $K_{\alpha}$ defined via $\widehat{K_{\alpha} f}(\xi) = |\xi|^{- \alpha} \widehat{f} (\xi)$ for {\color{red} $\alpha \in [0, 1/2)$}. An alternative representation that \CR{is} useful for the proof of Lemma~\ref{lem:part_new_operator} is
\begin{align}
\label{eq:ka_conv} K_{\alpha} f(x) = \int_R h(x - y) f(y) dy = \int_R k(x, y) f(y) dy\CR{,}
\end{align}
where $h(x - y) = k(x, y) = d_{\alpha} |x - y|^{\alpha -1}$ for some constant $d_{\alpha}$ (see e.g. \cite[Theorem 2.4.6]{Gr1}). 

\noindent Moreover, $K_{\alpha}$ can be extended to a bounded operator from $L^2(J)$ to $L^2(\mathbb{R})$ where $J \subset \mathbb{R}$ is compact.  
\CR{Let first  $f \in C_c^{\infty}(\mathbb{R})$ and  recall} that the smooth compactly supported functions are dense in $L^2(J)$.  \CR{Since $f$ is a Schwartz function,} $\widehat{f}$ is bounded and decays at any polynomial rate. Then $\int_{\mathbb{R} } |\widehat{K_{\alpha} f} (\xi)|^2 d\xi < \infty$, and by \CR{Plancherel's} theorem $\int_{\mathbb{R}} |K_{\alpha} f(x)|^2 dx < \infty$. 

 We are left to show that when $f \in L^2(J)$, $\int_{\mathbb{R} }| K_{\alpha} f(x)|^2 dx < C(J) ||f||_{L^2(J)}^2$, for some constant $C(J) > 0$ depending only on $J$.  By \CR{Plancherel's} theorem 
\CR{\begin{align}
\|K_\alpha f\|_{L^2(\R)}^2 =& \; \frac{1}{2\pi} \int_{|\xi| \leq 1} |\xi|^{ - 2 \alpha} |\widehat{f}(\xi)|^2 d \xi  + \frac{1}{2\pi}\int_{|\xi| >1} |\xi|^{ - 2 \alpha} |\widehat{f}(\xi)|^2 d \xi \notag\\
\label{eq:ka_1} \leq & \;\; \int_{|\xi| \leq 1} |\xi|^{ - 2 \alpha} |\widehat{f}(\xi)|^2 d \xi   + ||f||_{L^2(J)}^2, 
\end{align} }
To bound the first integral, note that by definition, for every $\xi \in \mathbb{R}$, 
\begin{align}
\label{eq:ka_2}  |\widehat{f}(\xi)| = \left|\int_{J} e^{- i \xi x} f(x) dx \right| \leq \int_{J} |f(x)| dx \leq |J|^{1/2} \left(\int_J |f(x)|^2\right)^{1/2} dx, 
\end{align}
\noindent where we have applied the Cauchy-Schwarz inequality, and $|J|$ is the Lebesgue measure of the compact set $J$. Therefore, combining~\eqref{eq:ka_1} and~\eqref{eq:ka_2}, yields
\begin{align}
\CR{\|K_\alpha f\|_{L^2(\R)}^2} \leq ||f||_{L^2(J)}^2 \left( |J| \int_{-1}^1 |\xi|^{- 2\alpha} d\xi + 1\right) = C(J) ||f||_{L^2(J)}^2 <  \infty, 
\end{align}
\noindent where $C(J) > 0$ is some constant depending only on $J$. This establishes that $K_{\alpha}$ extends to a bounded operator from $L^2 (J)$ to $L^2(\mathbb{R})$.

\CR{If $g$ is a bounded, compactly supported function and $M_g$ is the multiplication operator given by $M_g f (x) = g(x) f(x)$,  then $M_g:L^2(\R)\to L^2(\R)$ is obviously bounded and its (Hilbert space adjoint) equals $M_{\overline{g}}$. Our previous observation on $K_\alpha$ then implies  that $K_\alpha M_{g}$ extends to a bounded operator on $L^2(\R)$. Let us check that the adjoint of  $K_\alpha M_{g}$ equals  $M_{\overline{g}}K_\alpha $ (we need to be cautious since $K_\alpha$ is not bounded on the whole of $L^2(\R)$): for any $f,h\in C_0^\infty(\R)$ we obtain 
\begin{align}
\notag \int_{\mathbb{R}} K_{\alpha} M_g f (x)& \overline{h(x)} dx\; =\;\;  \frac{1}{2\pi}\int_{\mathbb{R}} \widehat{(K_{\alpha} M_g f)} (\xi) \overline{\widehat{h}(\xi)} d\xi \notag
\; = \; \frac{1}{2\pi}\int_{\mathbb{R}} \widehat{g} * \widehat{f} (\xi) \overline{|\xi|^{-\alpha} \widehat h(\xi)} d \xi\\
 =\; &\frac{1}{2\pi} \int_{\mathbb{R}}\int_{\mathbb{R}} \widehat{g}(\xi - t) \widehat{f}(t) dt \overline{\widehat{K_{\alpha} h}(\xi)}d \xi \;=\;
 \int_{\mathbb{R}} g(x)f(x)\overline{K_{\alpha} h(x)} dx\notag \\
\notag  = \;&
 \int_{\mathbb{R}} f(x) \overline{M_{\bar{g}} K_{\alpha} h(x)} dx.
\end{align} }
Therefore, $M_{\overline{g}}K_{\alpha}$ extends to a bounded operator from $L^2(\mathbb{R})$ to $L^2(\mathbb{R})$, and by extension, so does $K_{\alpha} M_g K_{\alpha}$, since $K_{\alpha} M_g K_{\alpha} = K_{\alpha} M_g (K_{\alpha} M_{\chi_J})^*$, where $J \subset \mathbb{R}$ is a compact \CR{interval} containing the support of $g$.

\begin{proof}[Proof of Lemma~\ref{lem:part_new_operator}] Let $\CR{T} : L^2(|y|^{-\CR{H}} dy) \to L^2 (\mathbb{R})$ be given by $(\CR{T} f)(x) = |x|^{-H/2} f(x)$. Note that $\CR{T}$ is an isometric isomorphism.  Hence, the operator $A_{t, \xi}$ is isometrically isomorphic to $V_{t, \xi} \coloneqq \CR{T} A_{t, \xi} \CR{T}^{-1} : L^2(\mathbb{R}) \to L^2(\mathbb{R})$, that satisfies
\begin{align}
\notag V_{t, \xi}f(x) = |x|^{-H/2} \int_{\mathbb{R}} \sum_{j = 1}^n  \xi_j \frac{e^{i t_j (x - y) } -1}{i (x - y)} f(y) |y|^{-H/2}dy.
\end{align}
\noindent Next, recall that  the Fourier transform $\mathcal{F}$ provides an isometric isomorphism $L^2(\mathbb{R}) \to L^2(\mathbb{R})$, up to a constant, and transforms multiplication to convolution. Since
\begin{align}
\notag \CR{(\mathcal{F} g)(x)} = \left( \sum_{j =1}^n \xi_j \mathcal{F} \chi_{[0, t_j]} \right) (x) = \sum_{j =1}^n \xi_j \frac{e^{-i t_j x} - 1}{- i x},
\end{align} 
\noindent it follows that
\begin{align}
\notag  \mathcal{F}^3  (  K_{H/2} M_g K_{H/2} f)(x) = & 2\pi  \mathcal{F}  (  K_{H/2} M_g K_{H/2} f)(-x)
\;=\; \CR{|x|^{-H/2} \left( (\mathcal{F} g) \ast \widehat{K_{H/2} f} \right) (-x)} \\ 
\notag  = &   %(2 \pi)^{-1}
2\pi |x|^{-H/2}  \int_{\mathbb{R}} \sum_{j = 1}^n  \xi_j \frac{e^{-i t_j (-x - y) } - 1 }{- i (-x - y)} \widehat{f}(y) |y|^{-H/2}dy, \\
\notag = &  %(2 \pi)^{-1} 
2\pi|x|^{-H/2}  \int_{\mathbb{R}} \sum_{j = 1}^n  \xi_j \frac{e^{i t_j (x - y) } - 1 }{i (x - y)} \widehat {f}(-y) |y|^{-H/2}dy, \\
\notag = &  %(2 \pi)^{-1} 
|x|^{-H/2}  \int_{\mathbb{R}} \sum_{j = 1}^n  \xi_j \frac{e^{i t_j (x - y) } - 1 }{i (x - y)} (\mathcal{F}^3{f})(y) |y|^{-H/2}dy, \\
\notag = & 
(V_{t, \xi} \mathcal{F}^3 f)( x). 
\end{align} 
\noindent where $\ast$ denotes convolution, and we have used that $\mathcal{F}^3 f (x) = 2\pi\widehat f(-x)$. Therefore, there is a constant $c(H)$ depending only on $H$, such that $A_{t, \xi}$ and  $B(t, \xi) \coloneqq c(H) K_{H/2} M_g K_{H/2}$ are unitarily equivalent. In particular, since $A_{t, \xi}$ is self-adjoint, so is $B(t, \xi)$, although this can be easily seen also from the definition of $B$. Moreover, the two operators have the same eigenvalues,  \CR{and, more importantly, their singular value sequences coincide}.

 In order to establish~\eqref{eq:bound_svalue} we need two technical results. First,  a key Lemma~\ref{le:piecing} compares the singular values of $B_{t, \xi}$ to $B_{t_1, \xi_1} = c(H) \xi_1 K_{H/2} \chi_{[0, t_1]} K_{H/2} $. Then, Lemma~\ref{le:jelppi}, establishes asymptotics for the singular values of a related operator. 
\begin{lemma}\label{le:piecing}
Let $\alpha\in (0,1/2)$ and $I\subset\R$ be an interval. Then, there is a bounded operator $U_{\alpha,I}$ on $L^2(\R)$, \CR{whose norm is bounded by a finite constant that depends only on $\alpha$}, so that for any compactly supported $f\in L^2(\R)$:
\begin{equation}\label{eq:U}
K_\alpha(M_{\chi_I}f)=U_{\alpha,I}K_\alpha f.
\end{equation}
\end{lemma}
\begin{proof} We have shown that if $f \in L^2(\mathbb{R})$ is compactly supported, then $K_{\alpha} f \in L^2(\mathbb{R})$. Therefore, we need to show that for all $h \in L^2(\mathbb{R})$, the operator $U_{\alpha, I}$, \CR{defined a priori only on $C_0^\infty(\R)$}
\begin{align}
\notag U_{\alpha,I}h \coloneqq K_\alpha M_{\chi_I}K_{-\alpha}h, 
\end{align}
extends to a well defined bounded operator from $L^2(\mathbb{R})$ to $L^2(\mathbb{R})$ \CR{(note that the factor $K_{-\alpha}$ alone does not have this property, but $K_{-\alpha}f\in L^{2}(\R)$ for $f\in C_0^\infty(\R))$. }
\noindent Again we use that $C_c^{\infty}(\mathbb{R})$ is dense in $L^2(\mathbb{R})$. Let $h \in C_c^{\infty} (\mathbb{R})$. Then, as before, by the Paley-Wiener theorem and the \CR{Plancherel's} theorem, $U_{\alpha, I}$ is a \CR{well-defined  operator} from $C_c^{\infty}(\R)$ to $L^2(\R)$, and~\eqref{eq:U} holds. To extend it to all of $L^2(\R)$, we show the following a priori norm bound 
\begin{equation}
\label{eq:U2} \int_\R |K_\alpha M_{\chi_I}K_{-\alpha}h |^2 \leq c(\alpha)\int_\R |h|^2
\end{equation}
for $h = K_{\alpha} f$, where \CR{$f \in C_0^\infty (\R)$ -- such functions are clearly dense in $L^2(\R).$} Note, that $f = K_{-\alpha} h$. From Parseval identity we get
\begin{align}
\notag \int_\R |K_\alpha M_{\chi_I}K_{-\alpha} f|^2 = \int_\R |K_\alpha M_{\chi_I} f|^2 = (2\pi)^{-1}\int_\R |\xi|^{-2\alpha}|\widehat{M_{\chi_I} f} (\xi)|^2 d\xi,
\end{align}
\noindent and, by the convolution theorem, 
\begin{align}
\notag \widehat{M_{\chi_I} f} =(2\pi)^{-1}\widehat{\chi_I} * \widehat{f},
\end{align}
so everything is well-defined as $\widehat f$ decays at any polynomial rate.
The right-hand side of~\eqref{eq:U2}  can be written as
\begin{align}
\notag \int_\R |h|^2 =\int_{\R} |K_\alpha f|^2 = \int_\R |\xi|^{-2\alpha}|\widehat{f}(\xi)|^2 d\xi.
\end{align}
To establish~\eqref{eq:U2}, we need to show
\begin{align}
\label{eq:U3} \int_\R |\xi|^{-2\alpha}|\widehat{M_{\chi_I} f} (\xi)|^2 d\xi \leq c(\alpha) \int_\R |\xi|^{-2\alpha}|\widehat{f}(\xi)|^2 d\xi.
\end{align}
We recall some properties of the Hilbert transform (see~\cite{Gr1},~\cite{Duo-2001}), \CR{defined (at least)  for test functions as the singular value integral}
\begin{align}
\notag \mathcal{H} g (x) \coloneqq \lim_{\ep \to 0} \frac{1}{\pi} \int_{|x - y | \geq \ep} \frac{g(y)}{x - y} dy,
\end{align}  
where $g \in L^2(\R)$. \CR{First, the Hilbert transform has a bounded extension to $L^2(\R)$. The boundedness is most easily seen by the fact that it is} a multiplier operator with bounded \CR{symbol:}
\begin{align}
\notag \mathcal{F} ( \mathcal{H} g ) (\xi) = ( - i \text{sgn}(\xi) ) \widehat{g}(\xi).
\end{align}
Moreover (see~\cite[Section 3.5]{Duo-2001}),
\begin{align}
\notag \mathcal{F} \left(\frac{i}{2} (N_a \mathcal{H} N_{-a} - N_b \mathcal{H} N_{-b} ) g\right)(\xi)   =  \chi_{(a,b)} (\xi) \widehat{g}(\xi),
\end{align} 
where $N_a$ is the isometric multiplication operator given by \CR{$N_a = M_{e^{i   a x}}$}. Applying $\mathcal{F}^{-1}$ to the above (and recalling that \CR{$\mathcal{F}^{-1} f (\xi)  = \frac{1}{2\pi}\widehat{f} (-\xi)$) yields} 
\begin{align}
\notag \left(\frac{i}{2} (N_a \mathcal{H} N_{-a} - N_b \mathcal{H} N_{-b} ) g\right)(\xi) & = \CR{\frac{1}{2\pi}} \int_{\R} \widehat{\chi_{(a,b)}} (-y) g(x -y) dy \\ 
\notag & = \CR{\frac{1}{2\pi} \widehat{\chi_{(-b,a)}} }* g (\xi).
\end{align} 
In particular,  if $I=(-b,-a)$, we obtain up to an absolute constant, 
\begin{align}
\notag \left|\left(\frac{i}{2} (N_a \mathcal{H} N_{-a} - N_b \mathcal{H} N_{-b} ) \widehat{f}\right)(\xi)\right|^2 = \CR{c}|\widehat{M_{\chi_I} f }(\xi)|^2.
\end{align}
Then, since $N_a$ is isometric \CR{in the weighted space}, in order to establish~\eqref{eq:U3} it suffices to show that 
\begin{align}
\label{eq:U4} \int_{\R} |\mathcal{H} \widehat{f}(\xi)|^2 |\xi|^{-2\alpha} d \xi \leq c(\alpha) \int_{\R} |\widehat{f}(\xi)|^2 |\xi|^{-2 \alpha} d \xi. 
\end{align}  
At this point we recall a more general result due to Hunt et al~\cite{Hunt-1973}.
\begin{proposition} If $1 < p < \infty$ and $W(x)$ is nonnegative, the following are equivalent. 
\begin{enumerate}
\item There is a constant $C$ independent of $I$, such that for every interval $I$, 
\begin{align}
\label{eq:ap_weight} \left[ \frac{1}{|I|} \int_I W(x) dx \right] \left[ \frac{1}{|I|} \int_I W(x)^{-1 / (p-1)}dx\right]^{p-1} \leq C.
\end{align}
\item There is a constant $C$, independent of $f$, such that 
\begin{align}
\label{eq:hil_norm_weight} \int_{- \infty}^{\infty} |\mathcal{H} f (x)|^p W(x) dx \leq C \int_{- \infty}^{\infty} |f(x)|^p W(x) dx. 
\end{align}
\end{enumerate}
\end{proposition}
Nonnegative functions $W(x)$ that satisfy~\eqref{eq:ap_weight} are called $A_p$ weights. We can apply~\eqref{eq:hil_norm_weight} in order to establish~\eqref{eq:U4} provided that $|x|^{-2 \alpha}$ is an $A_2$ weight. This fact is established for $\alpha \in (0, 1/2)$ in, e.g.,~\cite[Example 9.1.7.]{Gr2}

\end{proof}

\begin{remark}\label{rem:operator}\CR{\rm The previous lemma is crucial in our estimation of the singular values, since it verifies that the singular values will dominate those of any \emph{localization } of the multiplier $g$. This is not at all obvious, since $g$ typically changes sign in our situation, and a potential cancellation phenomenon could prevent the needed estimate.}

\end{remark}

We now state the second technical lemma. 
\begin{lemma}\label{le:jelppi} Let $\alpha\in (0,1/2)$ and  $J\subset \R$ be a finite interval. Then the operator $\mathcal{M}_{\alpha, J}: = M_{\chi_J}K_\alpha M_{\chi_J}$ is \CR{bounded, self-adjoint, and positive},  and its singular value sequence is of the form  $(|J|^\alpha\widetilde \mu_k)_{k\geq1}$, where $(\widetilde \mu_k)_{k\geq 1}$ is the singular value sequence of $\mathcal{M}_{\alpha, [0,1]}$. Moreover, $\widetilde \mu_k>0$ for all $k\geq 1$, and 
\begin{align}
\notag\widetilde\mu_k\sim ck^{-\alpha}\quad\textrm{as}\quad k\to\infty.
\end{align}
\end{lemma}
\begin{proof} Recall that $K_{\alpha} M_{\chi_J}$ is bounded since $J$ is a finite interval, and note that$M_{\chi_J}$ is an orthogonal projection and  has norm $1$, especially $M_{\chi_J}^2=M_{\chi_J}$ and 
$M_{\chi_J}^*=M_{\overline{\chi_J}}=M_{\chi_J}.$ Thus $\mathcal{M}_{\alpha, J}$ is bounded. Finally, by \CR{Plancherel's} theorem, for any $f,g\in L^2(\R)$
\begin{align}
\notag \int_\R (M_{\chi_J}K_\alpha M_{\chi_J}f)(x) \CR{\overline{g(x)}}dx  = & \int_\R (K_\alpha M_{\chi_J}f)(x) \CR{\overline{M_{\chi_J}g(x)}} dx \\
\notag = &\int_\R |\xi|^{-\alpha} \widehat{M_{\chi_J}f} (\xi)  \CR{\overline{\widehat{M_{\chi_J}g} (\xi)} }d\xi.
\end{align}
This shows that $\mathcal{M}_{\alpha, J}$ is self-adjoint and positive (take $g = f$ above). 

To establish the relation of the singular values of $\mathcal{M}_{\alpha, J}$ to the ones of $\mathcal{M}_{\alpha, J}$ we will use the fact that the spectrum of two unitarily equivalent operators is identical. First, assume $J = [b, c]$ and let $\tau_b f (x) \coloneqq f(x -b)$ be the translation by $b$. The operator $\tau_b$ is unitary with an adjoint given by $\tau_{-b} f(x) = f(x + b)$. Then, note that
\begin{align}
\notag \tau_{-b} \mathcal{M}_{\alpha, [b,c]} \tau_b f (x)  = & \chi_{[b,c]} (x + b) \int_\R h( x + b - y) \chi_{[b,c]} (y) f(y - b) dy \\
\notag = & \chi_{[0, c-b]} (x) \int_\R h(x - y') \chi_{[0, c-b]} (y') f(y') dy' \\
\notag = & \mathcal{M}_{\alpha, [0, c-b]}(x).
\end{align} 
Therefore, the singular values of $\mathcal{M}_{\alpha, [b,c]}$ are the same as the ones of $\mathcal{M}_{\alpha, [0, c-b]}$. Next, set $J = [0, b]$ and consider the rescaling unitary operator $\mu_b f (x) \coloneqq \sqrt{b} f(bx)$ with an inverse given by $\CR{\mu_{b}^{-1}} f(x) = f(x/b) / \sqrt{b}$. Now,
\begin{align}
\notag \mu_{b} \mathcal{M}_{\alpha, [0,b]}\CR{ \mu_{b}^{-1}} f (x)  = & \CR{\sqrt{b}}\chi_{[0,b]} (bx)  \int_\R h( bx - y) \chi_{[0,b]} (y) f(y/b)/\sqrt{b} dy \\
\notag = & \chi_{[0, 1]} (x) \int_\R h(b(x - y')) \chi_{[0, 1]} (y') f(y')\CR{b}  dy' \\
\notag = & b^{\alpha} \mathcal{M}_{\alpha, [0, 1]}(x),
\end{align} 
where we have used that $h(bz) = b^{\alpha -1} h(z)$. Thus, the correspondence between the singular values of $\mathcal{M}_{\alpha, J} $ and $\mathcal{M}_{\alpha, [0,1]}$ is established. 

Next, for the proof of the decay of the singular values of $\mathcal{M}_{\alpha, [0,1]}$, we need to show that $\mathcal{M}_{\alpha, J}$ is a compact operator. Recall~\eqref{eq:ka_conv}:
\begin{align}
\notag K_{\alpha} f (x)  = \int_\R h(x - y) f(y) dy = \int_\R k(x, y) f(y) dy. 
\end{align}  
Integral operators are compact if $\int_{\CR{\R^2}} |k(x, y)|^2 dx dy < \infty$. However, for the operator $\mathcal{M}_{\alpha, J}$ one only needs to show that  $\int_{\CR{J^2}} |k(x, y)|^2 dx dy < \infty$ which is indeed the case for $k(x, y) = d_{\alpha} |x- y|^{\alpha -1}$ \CR{with $\alpha>1/2.$}

To finish the proof, recall a result by Dostanic~\cite[Theorem 1]{Dostanic-1998} where the decay of the singular values is established for $ \mathcal{M}_{\alpha, [-1,1]} / d_{\alpha}$.
\begin{proposition} Let $A : L^2(-1,1) \to L^2 (-1,1)$ be the \CR{self-adjoint operator} defined by
\begin{align}
\notag Af (x) = \int_{-1}^1 |x - y|^{\alpha -1} f(y) dy, \quad \quad 0 < \alpha < 1.
\end{align}
Then, the eigenvalues of $A$  \CR{are simple} and satisfy $\lambda_n (A) 	\sim cn^{-\alpha}$ with a constant $c=c(alpha)>0$.
\end{proposition}

\end{proof}
Next, we finish the proof of~\eqref{eq:bound_svalue}. We have $t_0 = 0 < t_1 < \cdots < t_n  \leq 1$. Set $I_j=[t_{j-1},t_j]$ for $1 \leq j \leq n$. Then $g(x) = \sum_{j =1}^n (\xi_j - \xi_{j-1} )\chi_{I_j}$. Fix $j \in [1, n]$ and let $  \CR{ (\|U_{\alpha,I_{j}}\|)^{-1}\geq C(\alpha)>0}$ with $U_{\alpha,I_{j}}$ as in Lemma~\ref{le:piecing}. Recall that if $A$ and $B$ are bounded operators on $L^2(\R)$, then by the Minimax principle~\cite[Corollary III.1.2]{Bhatia-1997}, 
\begin{align}
\notag \mu_n(AB) = \max_{\substack{M \subset L^2(\R) \\ \dim M = n}} \min_{\substack{ x \in M \\ ||x|| = 1}}\CR{\| AB x\|}
\end{align} 
and then $\mu_n(AB) \leq \Vert A\Vert \mu_n(B)$ and $\mu_n(AB) \leq \Vert B\Vert \mu_n(A)$. Therefore,
\begin{align}
\notag \mu_n(K_\alpha M_gK_\alpha) \geq (\|U_{\alpha,I_{j}}\|)^{-1} \mu_n(U_{\alpha,I_{j}}K_\alpha M_gK_\alpha)= C(\alpha) \mu_n(K_\alpha M_{\chi_{I_{j}}}M_gK_\alpha).
\end{align}
Moreover, 
\begin{align}
\notag \mu_n(K_\alpha M_{\chi_{I_{j}}}M_gK_\alpha) = \mu_n(K_\alpha ( (\xi_{j} - \xi_{j-1}) M_{\chi_{I_{j}}})K_\alpha) =|\xi_{j} - \xi_{j-1}|\mu_n(K_\alpha M_{\chi_{I_{j}}}K_\alpha). 
\end{align}
Then, since $\Vert M_{\chi_{I_{j}}} \Vert = 1$ and $ M_{\chi_{I_{j}}} ^2= M_{\chi_{I_{j}}} $ we have
\begin{align}
\notag \mu_n(K_\alpha M_{\chi_{I_{j}}}K_\alpha) \geq \mu_n(M_{\chi_{I_{j}}} K_\alpha M_{\chi_{I_{j}}}K_\alpha M_{\chi_{I_{j}}}) = \mu_n\left((M_{\chi_{I_{j}}}K_\alpha M_{\chi_{I_{j}}})^2\right). 
\end{align}
Next, by Lemma~\ref{le:jelppi} 
\begin{align}
\notag \mu_n\left((M_{\chi_{I_{j}}}K_\alpha M_{\chi_{I_{j}}})^2\right) = \mu_n(M_{\chi_{I_{j}}}K_\alpha M_{\chi_{I_{j}}})^2 =  |t_j - t_{j-1}|^{2 \alpha} (\widetilde \mu_n)^2,
\end{align}
where we have used that if $A$ is self-adjoint, $\mu_n(A^2) = \mu_n (A)^2$. Therefore, 
\begin{align}
\notag \mu_n(K_\alpha M_gK_\alpha) \geq C(\alpha) \max_{1 \leq j \leq n} |\xi_{j} - \xi_{j-1}| |t_j - t_{j-1}|^{2 \alpha} (\widetilde \mu_n)^2,
\end{align} 
where $\widetilde \mu_n \sim n^{-\alpha}$. Finally,~\eqref{eq:bound_svalue} follows with $\alpha = H/2$. 
\end{proof}

\subsection{Integral estimates}
\begin{proof}[Proof of Lemma~\ref{lem:int_svalue}] Note, that for any $N > 0$, 
\begin{align}
\notag 1 + \sum_{k = 1}^N 4s^2 \widetilde \mu_k^4 \leq \prod_{k = 1}^N (1 + 4s^2 \widetilde \mu_k^4) \leq \CR{\exp( \sum_{k=1}^N 4s^2 \widetilde \mu_k^4 ).}
\end{align}
By Lemma \ref{le:jelppi}, $\widetilde\mu_k^4\; \CR{\approx} \; k^{-2H}$ and  $\sum_{k = 1}^{\infty} 4s^2 \widetilde \mu_k^4  < \infty$ since $H \in (1/2 ,1 )$. Therefore, $G(s)$ converges and also $G(s) > 0$.
Next, let $ a > 0$ be such that $4\widetilde\mu_k^4 \geq a k^{-2H}$ for all $k \geq 1$ and set $z = as^2$. Then, the elementary inequality $\log(1+x)\geq x/2$ for $x\in [0,1)$ yields for $z\geq 1$ that
\begin{eqnarray*}
-8\log G(s)&=&2\sum_{k=1}^\infty \log (1+ 4s^2\widetilde \mu_k^4)\geq  2\sum_{k=1}^\infty \log (1+ zk^{-2H})\geq z\sum_{ k\geq z^{1/2H}}k^{-2H}\\
&\geq& z\int_{2z^{1/2H}}x^{-2H}= 2^{1-2H}(2H-1)^{-1}z^{1/2H}\geq c_0z^{1/2H},
\end{eqnarray*}
where $c_0 > 0$, and in general $-8\log G(s)\geq c_0(z^{1/2H}-1) $.
Thus, $G(s)\leq c_2\exp(-c_1s^{1/H})$, where $c_1, c_2 > 0$, and we obtain 
\begin{align}
\notag \int_0^\infty s^{\beta-1}G(s)ds \leq & \int_0^\infty  c_2\exp(-c_1s^{1/H})s^{\beta-1}ds 
\;= \; c_2 H \int_0^{\infty} \exp( -c_1 x) x^{\beta H - 1} dx \\
\notag = & c_2 H c_1^{- \beta H} \Gamma(\beta H) \CR{\;\leq\; c_3^{\beta H} \Gamma(\beta H) }.
\end{align}
as desired.

\end{proof} 

\begin{proof}[Proof of Lemma~\ref{lem:alku}] Our goal is to estimate the integral 
 \begin{align}
\notag I_0 \coloneqq \int_{S^{n-1}}\int_{\substack{t_1+t_2+\ldots +t_n\leq 1\\ t_1,\ldots, t_n\geq 0}}\prod_{j=1}^n |y_j|^{\gamma_j}(f(t,y))^{-n(1+\gammav)}dt\mathcal{H}^{n-1}(dy),
 \end{align}
where $f: \R_+^n \times \R^n \to (0, \infty)$ is given by
\begin{align}
\notag f(t, y) = t_1^H |y_1| \vee t_2^H |y_2| \vee \cdots \vee t_n^H |y_n|,
\end{align}
and $\gammav:=n^{-1}\sum_{j=1}^n\gamma_n$ for $\gamma_1,\ldots \gamma_n \geq 0$.
First note that $f$ is $H$-homogeneous with respect to $t$ and 1-homogeneous with respect to $y$:
 \begin{equation}\label{eq:ehto}
 f(t,y)=|t|^H |y|f(t^0,y^0),
 \end{equation}
 where $y^0\coloneqq y/|y|$ and $t^0 \coloneqq t/|\sum_{j=1}^n t_j|$. Observe that we use the standard Euclidean norm for the $y$-variable and the $\ell^1$-norm for the $t$-variable.
Next, introduce the related integral
 \begin{align}
 \notag I_1\coloneqq \int_{S^{n-1}}\int_{\substack{t_1+t_2+\ldots +t_n=1\\ t_1,\ldots, t_n\geq 0}}\prod_{j=1}^n |y_j|^{\gamma_j}(f(t,y))^{-n(1+\gammav)}\mathcal{H}^{n-1}(dt)\mathcal{H}^{n-1}(dy),
\end{align}
\noindent 
 Using the $H$-homogeneity  of $f$ with respect to $t$ and noting that the distance of the origin from the hyperplane $t_1+\ldots + t_n=1$ equals $n^{-1/2}$, we obtain the following relation between $I_0$ and $I_1$:
\begin{align}
\notag I_0=& I_1\int_0^{n^{-1/2}} (u/n^{-1/2})^{-n(1+\gammav)H}(u/n^{-1/2})^{n-1}du,
\end{align}
where the term $(u/n^{-1/2})^{n-1}$ arises from the Jacobian and the rescaling. The above can be further simplified:
\begin{align}
\label{eq:kaali} I_0  = I_1n^{-1/2}\int_0^1v^{n(1-H(1+\gammav))-1}dv  = (1-H(1+\gammav))^{-1}n^{-3/2}I_1,
\end{align}
where one naturally needs to assume that $H(1+\gammav)<1$.

We proceed by showing the following more general result. 
 \begin{lemma}\label{le:helppi10}
 Assume that $f:\R_+^n\times\R^n\to (0,\infty)$ satisfies \eqref{eq:ehto}.  Let $H\in (0,1)$ and  $\gamma \coloneqq (\gamma_1, \ldots , \gamma_n) \in \R$ with $0\leq \gamma_j  <H^{-1}-1$  for every $j=1,\ldots , n$. Set $\gammav:=n^{-1}\sum_{j=1}^n\gamma_j$. Then
 \begin{align}
 \label{eq:jelppi100} I_0 =  C(n, \gamma, H)  \int_{(\R_+)^n\times\R^n}e^{-f(t,y)}e^{-(t_1+\ldots +t_n)}\prod_{j=1}^n|y_j|^{\gamma_j}dydt
 \end{align}
where
\begin{align}
\notag C(n, \gamma, H) \coloneqq \frac{1}{n^{1/2}\Gamma (n(1+\gammav))\Gamma(n(1-H(1+\gammav)+1)}.
\end{align}
Moreover, if $f$ is such that
 \begin{align}
\notag  f(y,t)=g(t_1^H |y_1|,\ldots, t_n^H |y_n|),
 \end{align}
where $g$ is $1-$homogeneous, then
  \begin{equation}\label{eq:jelppi200}
I_0
 =  C(n, \gamma, H) \prod_{j=1}^n\Gamma(1-H(1+\gamma_j)) \int_{\R^n}e^{-g(|y_1|,\ldots, |y_n|)}|y_1|^{\gamma_1}\cdots |y_n|^{\gamma_n}dy.
 \end{equation}
\end{lemma}
\begin{proof} We  first compute $I_1$ by
moving to radial variables in $y$ and $t$. Thus,  let $r:=|y|$, $w:=y/r,$ $s:=t_1+t_2+\ldots + t_n$ and $u:=t/s$ and   note as a first step that
\begin{eqnarray*}
&&\int_{(\R_+)^n\times\R^n}e^{-f(t,y)}e^{-(t_1+\ldots +t_n)}|y_1|^{\gamma_1}\cdots |y_n|^{\gamma_n}dydt\\
&=&\int_{|w|=1} \int_{\substack{u_1+u_2+\ldots +u_n=1\\ u_1,\ldots, u_n\geq 0}} \int_0^\infty\int_0^\infty e^{-rs^H f(u,w)} e^{-s} |w_1|^{\gamma_1}\cdots |w_n|^{\gamma_n} r^{n \gammav} \\
&&\phantom{kukkuukukkuukaukana}s^{n-1}r^{n-1}dr ds\mathcal{H}^{n-1}(du)\mathcal{H}^{n-1}(dw)\nonumber
\end{eqnarray*}
We compute first the integral with respect to $r$ by making a change of variables $r':=rs^H f(u,t)$ and then with respect to $s$ to obtain
\begin{eqnarray*}
&&\int_0^\infty\int_0^\infty e^{-rs^H f(u,w)}e^{-s}s^{n-1}dsr^{n(1+\gammav)-1}dr\\
& = & \int_0^\infty e^{-s} (f(u,w))^{- n ( 1 + \gammav)} s^{n (1 - H ( 1 + \gammav) )-1} \int_0^\infty e^{-r'} (r')^{n ( 1 + \gammav ) -1} dr' ds  \\
&=& \Gamma(n(1+\gammav)) (f(u,w))^{-n(1+\gammav)} \int_0^\infty e^{-s}s^{n(1-H(1+\gammav))-1}ds\\& = &\Gamma(n(1+\gammav))\Gamma ((1-H(1+\gammav))n) (f(u,w))^{-n(1+\gammav)} .
\end{eqnarray*}
Next, recall that $I_1$ is given by
\begin{align} 
\notag  I_1 = \int_{|w|=1} \int_{\substack{u_1+u_2+\ldots +u_n=1\\ u_1,\ldots, u_n\geq 0}} |w_1|^{\gamma_1}\cdots |w_n|^{\gamma_n} (f(u,w))^{-n(1+\gammav)}\mathcal{H}^{n-1}(du)\mathcal{H}^{n-1}(dw).
\end{align} 
Therefore,
\begin{align}
\notag & \int_{(\R_+)^n\times\R^n}e^{-f(t,y)}e^{-(t_1+\ldots +t_n)}|y_1|^{\gamma_1}\cdots |y_n|^{\gamma_n}dydt \\
\notag = & \Gamma(n(1+\gammav))\Gamma ((1-H(1+\gammav))n)  I_1\\
\notag  = &  \Gamma(n(1+\gammav))\Gamma ((1-H(1+\gammav))n)    (1-H(1+\gammav)) n^{3/2} I_0 \\
\notag = & C(n, \gamma, H)^{-1} I_0,
\end{align} 
where we have used  the relation \eqref{eq:kaali}.

To establish~\eqref{eq:jelppi200} note that with the change of variables $y_j = u_j / t_j^H$,   
\begin{align}
\notag & \int_{(\R_+)^n\times\R^n}e^{-g(\CR{t_1^H |y_1|,\ldots, t_n^H |y_n|})}e^{-(t_1+\ldots +t_n)}|y_1|^{\gamma_1}\cdots |y_n|^{\gamma_n}dydt \\
\notag = & \int_{(\R_+)^n\times\R^n}e^{-g(\CR{|u_1|, \ldots, |u_n|})}e^{-(t_1+\ldots +t_n)}|u_1|^{\gamma_1}\cdots |u_n|^{\gamma_n} \prod_{j = 1}^n t_j^{- H(1 + \gamma_j)} dudt \\
\notag = &\prod_{j = 1}^n \Gamma(1 - H(1 + \gamma_j))  \int_{\R^n}e^{-g(\CR{|u_1|, \ldots, |u_n|)}} |u_1|^{\gamma_1}\cdots |u_n|^{\gamma_n} \prod_{j = 1}^n t_j^{- H(1 + \gamma_j)} du,
\end{align}
and the conclusion follows after an application of~\eqref{eq:jelppi100}. 
\end{proof}

Finally, we establish~\eqref{eq:alku}.  By homogeneity, 
\begin{align}
\notag \int_{\partial([-u,u])}\prod_{j=1}^n |y_j|^{\gamma_j}\mathcal{H}^{n-1}(dy) = u^{n \gammav} u^{n-1} \int_{\partial([-1,1]^n)}\prod_{j=1}^n|y_j|^{\gamma_j}\mathcal{H}^{n-1}(dy).
\end{align}
Next, by Fubini theorem, 
\begin{align}
\notag \int_{[-1,1]^n}\prod_{j=1}^n|y_j|^{\gamma_j}dy=\frac{1}{n(1+\gammav)}
\int_{\partial([-1,1]^n)}\prod_{j=1}^n|y_j|^{\gamma_j}\mathcal{H}^{n-1}(dy) \eqqcolon \frac{A}{n ( 1 + \gammav)},
\end{align}
where $A$ is the integral over the boundary. Moreover, 
\begin{align}
\notag \int_{[-1,1]^n}\prod_{j=1}^n|y_j|^{\gamma_j}dy = \prod_{j =1 }^n \int_{[-1, 1]} |y_j|^{\gamma_j} dy_j = \prod_{j =1}^n \frac{2}{1 + \gamma_j}.
\end{align}
Next, consider
\begin{align}
\notag &  \int_{\R^n}e^{-\max_{1\leq j\leq n}|y_j|}\prod_{j=1}^n|y_j|^{\gamma_j}dy  \\
\notag = & \int_0^\infty   \int_{\partial([-u,u])}e^{-\max_{1\leq j\leq n}|y_j|} \prod_{j=1}^n |y_j|^{\gamma_j}\mathcal{H}^{n-1}(dy) du \\
\notag = & \int_0^\infty   e^{-u} \int_{\partial([-u,u])} \prod_{j=1}^n |y_j|^{\gamma_j}\mathcal{H}^{n-1}(dy) du \\
\notag = & \int_0^\infty e^{-u}A u^{n-1}u^{n\gammav}du \\
\label{eq:jelppi300} = & n(1+\gammav)\Gamma (n(1+\gammav))\prod_{j=1}^n\frac{2}{1+\gamma_j}.
\end{align}
Finally, applying~\eqref{eq:jelppi300} in~\eqref{eq:jelppi200} with $g(|y_1|, \ldots , |y_n|) = \max_{1\leq j\leq n}|y_j|$ yields
\begin{align}
\notag  & \int_{S^{n-1}}\int_{\substack{t_1+t_2+\ldots +t_n\leq 1\\ t_1,\ldots, t_n\geq 0}}\prod_{j=1}^n |y_j|^{\gamma_j}(f(t,y))^{-n(1+\gammav)}dt\mathcal{H}^{n-1}(dy) \\
\notag = &  \frac{n^{1/2}(1+\gammav)}{\Gamma(n(1-H(1+\gammav)+1)} \prod_{j=1}^n \left[ \frac{2}{1+\gamma_j} \Gamma(1 - H(1 + \gamma_j))\right] 
\end{align}
where $f: \R_+^n \times \R^n \to (0, \infty)$ is given by
\begin{align}
\notag f(t, y) = t_1^H |y_1| \vee t_2^H |y_2| \vee \cdots \vee t_n^H |y_n|,
\end{align}
and~\eqref{eq:alku} is established.

\end{proof}

\noindent
{\bf Acknowledgments.}  G. Kerchev and I. Nourdin are supported by the FNR OPEN grantAPOGee at Luxembourg University.

\end{document}